\documentclass{amsart}
\usepackage{amssymb}
\usepackage{hyperref}
\usepackage{pgfplots}
\usepackage{mathrsfs}
\newtheorem{theorem}{Theorem}[section]

\newtheorem{definition}[theorem]{Definition}
\newtheorem{conjecture}[theorem]{Conjecture}
\newtheorem{lemma}[theorem]{Lemma}
\numberwithin{equation}{section}
\DeclareMathOperator{\dist}{dist}
\DeclareMathOperator{\supp}{supp}
\DeclareMathOperator{\spn}{span}
\DeclareMathOperator{\sgn}{sgn}

\DeclareFontFamily{U}{mathx}{\hyphenchar\font45}
\DeclareFontShape{U}{mathx}{m}{n}{
      <5> <6> <7> <8> <9> <10>
      <10.95> <12> <14.4> <17.28> <20.74> <24.88>
      mathx10
      }{}
\DeclareSymbolFont{mathx}{U}{mathx}{m}{n}
\DeclareFontSubstitution{U}{mathx}{m}{n}
\DeclareMathAccent{\widecheck}{0}{mathx}{"71}
\DeclareMathAccent{\wideparen}{0}{mathx}{"75}

\title[Restricted families of projections onto planes]{Restricted families of projections onto planes: The general case of nonvanishing  geodesic curvature}
\author[T.~L.~J.~Harris]{Terence L.~J.~Harris}
\address{Department of Mathematics, Cornell University, Ithaca, NY 14853-4201, USA}
\email{tlh236@cornell.edu}
\subjclass[2020]{42B10; 28E99}
\keywords{Orthogonal projections, Hausdorff dimension, Fourier transform}

\begin{document}
\begin{abstract} \begin{sloppypar} It is shown that if $\gamma: [a,b] \to S^2$ is $C^3$ with $\det(\gamma, \gamma', \gamma'') \neq 0$, and if $A \subseteq \mathbb{R}^3$ is a Borel set, then $\dim \pi_{\theta} (A) \geq \min\left\{ 2,\dim A, \frac{ \dim A}{2} + \frac{3}{4} \right\}$ for a.e.~$\theta \in [a,b]$, where $\pi_{\theta}$ denotes projection onto the orthogonal complement of $\gamma(\theta)$ and ``$\dim$'' refers to Hausdorff dimension. This partially resolves a conjecture of Fässler and Orponen in the range $1< \dim A \leq 3/2$, which was previously known only for non-great circles. For $3/2 < \dim A < 5/2$ this improves the known lower bound for this problem. \end{sloppypar} 
\end{abstract}
\maketitle

\section{Introduction}
Let $S^2$ be the unit sphere in $\mathbb{R}^3$. Given a curve $\gamma: [a,b] \to S^2$ and $\theta \in [a,b]$, let $\pi_{\theta}$ be the orthogonal projection onto $\gamma(\theta)^{\perp} \subseteq \mathbb{R}^3$, given by
\[ \pi_{\theta}(x) = x - \langle x, \gamma(\theta) \rangle \gamma(\theta), \quad x \in \mathbb{R}^3. \]
Let $\dim A$ denote the Hausdorff dimension of a set $A \subseteq \mathbb{R}^3$.
\begin{theorem} \label{mainthm} Let $\gamma: [a,b] \to S^2$ be $C^3$ with $\det( \gamma, \gamma', \gamma'')$ nonvanishing. If $A \subseteq \mathbb{R}^3$ is an analytic set, then 
\[ \dim \pi_{\theta}(A) \geq \min\left\{ 2,\dim A, \frac{ \dim A}{2} + \frac{3}{4} \right\}, \]
for a.e.~$\theta \in [a,b]$.  \end{theorem} 
This partially resolves Conjecture~1.6 from \cite{fasslerorponen14} in the range $\dim A \leq 3/2$, for projections onto planes. In the range $1 < \dim A \leq 3/2$ this was previously known in the special case of non-great circles; due to Orponen and Venieri~\cite{venieri}. In the range $3/2 < \dim A < 5/2$, Theorem~\ref{mainthm} improves and generalises the previous best known lower bound from \cite{harris20}, which was also specific to non-great circles. 

Denote the $s$-dimensional Hausdorff measure in Euclidean space by $\mathcal{H}^s$. In $\mathbb{R}^2$ the classical Marstrand projection theorem \cite{marstrand2} states that if $A \subseteq \mathbb{R}^2$ is a Borel set and $P_e$ denotes orthogonal projection onto the 1-dimensional subspace through $e \in S^1$, then for $\dim A \leq 1$, 
\[ \dim P_e (A) = \dim A, \qquad \text{$\mathcal{H}^1$-a.e.~} e \in S^1, \]
and for $\dim A >1$, 
\[ \mathcal{H}^1(P_e (A)) >0, \qquad \text{$\mathcal{H}^1$-a.e.~} e \in S^1. \]
This was generalised to higher dimensions by Mattila. In $\mathbb{R}^3$ there are two versions of the Martrand-Mattila projection theorem; one for lines and one for planes. The version for lines is analogous to the above with $S^1$ replaced by $S^2$. For planes, it states that if $A \subseteq \mathbb{R}^3$ is a Borel set, then if $\dim A \leq 2$,
\begin{equation} \label{marstrand} \dim \pi_v (A) = \dim A, \qquad \text{$\mathcal{H}^2$-a.e.~$v \in S^2$,} \end{equation}
and if $\dim A >2$, then 
\[ \mathcal{H}^2(\pi_v(A))>0, \qquad \text{$\mathcal{H}^2$-a.e.~$v \in S^2$,} \]
where $\pi_v$ denotes projection onto $v^{\perp}$. Restricted projection families can be formed by constraining $v$ to move along a one-dimensional curve $\gamma: [a,b] \to S^2$, and the restricted projection problem asks whether \eqref{marstrand} still holds with a natural 1-dimensional measure replacing the surface measure $\mathcal{H}^2$ on $S^2$. 

Without the assumption that $\det(\gamma, \gamma', \gamma'')$ is nonvanishing, the equality 
\begin{equation} \label{ledrap} \dim \pi_{\theta} (A) = \dim A, \quad \text{a.e.~$ \theta \in [a,b]$,} \end{equation}
can only hold (in general) for $\dim A \leq 1$, and was proved in this range by Järvenpää, Järvenpää, Ledrappier, and Leikas~\cite{jarvenpaa1} using the energy method of Kaufman \cite{kaufman}. Considering the counterexample where $\gamma$ is a great circle contained in a plane $P$ and $A \subseteq P$ shows that some extra assumption on $\gamma$ is necessary for \eqref{ledrap} to hold in general for $1< \dim A \leq 2$. The following conjecture is due to Fässler and Orponen; $\rho_{\theta}$ denotes projection onto the 1-dimensional subspace through $\gamma(\theta)$.
\begin{conjecture}[{\cite[Conjecture~1.6]{fasslerorponen14}}] Let $\gamma: [a,b] \to S^2$ be a $C^3$ curve with $\det(\gamma, \gamma', \gamma'')$ nonvanishing, and let $A \subseteq \mathbb{R}^3$ be an analytic set. Then 
\[ \dim \rho_{\theta} (A) = \min\{ \dim A, 1\}, \quad \text{a.e.~$\theta \in [a,b]$,} \]
and
\[ \dim \pi_{\theta} (A) = \min\{ \dim A, 2\}, \quad \text{a.e.~$\theta \in [a,b]$.} \] \end{conjecture} 
\begin{sloppypar} For projections onto planes, progress was made by Fässler-Orponen \cite{fasslerorponen14}, Oberlin-Oberlin~\cite{oberlin}, Orponen~\cite{orponen}, Orponen-Venieri~\cite{venieri}, and by the author in \cite{harris20}. The condition $\dim A>5/2$ remains the best known sufficient condition that ensures $\mathcal{H}^2(\pi_{\theta}(A)) > 0$ for a.e.~$\theta \in [a,b]$; this is due to Oberlin-Oberlin~\cite{oberlin}. A comparison with some of the previous bounds is shown in Figure~\ref{picture}. \end{sloppypar}

\begin{figure} 
\begin{tikzpicture}[
declare function={
	  func(\x)= (\x<=1.5) * (\x)   +
              and(\x>1.5, \x<2.25) * (\x*0.444444+ 0.833333) +
							and(\x>=2.25, \x<2.5) * (\x*0.666666+ 0.333333);
		func2(\x)= (\x<=1.5) * (\x)   +
              and(\x>1.5, \x<=3) * (1+\x*0.333333);
		func3(\x)= (\x<=2) * (\x)   +
              and(\x>2, \x<=3) * (2);
		func4(\x)= (\x<=2) * (0.75*\x)   +
              and(\x>2, \x<=2.5) * (\x-0.5)		+
							and(\x>2.5, \x<=3) * (2);
		func5(\x)= (\x<=1.5) * (\x)   +
              and(\x>1.5, \x<=2.5) * (\x*0.5 + 0.75)		+
							and(\x>2.5, \x<=3) * (2);
  }]
  \begin{axis}[ 
	grid=major,,
	xtick distance=0.25,
    xlabel=$\dim A$,
    ylabel={$\dim \pi_{\theta} (A)$},
		legend style={
legend pos=outer north east,
}
  ] 
 \addplot[densely dotted,thick,domain=2.25:2.5]{func4(x)}; 
 \addplot[dashed,thick,domain=1:2.25]{func2(x)}; 
 \addplot[black,thick,domain=1.5:2.5]{func(x)}; 
 \addplot[red,thick,opacity = 0.4, domain=1:2.5]{func5(x)};
 \addplot[black,thick,opacity = 0.3, domain=1.5:2.5]{func3(x)};
\legend{Oberlin-Oberlin, Orponen-Venieri, \cite{harris20}, Theorem~\ref{mainthm}, Conjecture}
  \end{axis}
\end{tikzpicture}
\caption{The current and some of the previous a.e.~lower bounds for $\dim \pi_{\theta}(A)$, in  the range $\dim A \in (1,5/2)$. The Orponen-Venieri theorem and the bound from \cite{harris20} were for the special case of non-great circles.}
\label{picture}
\end{figure}
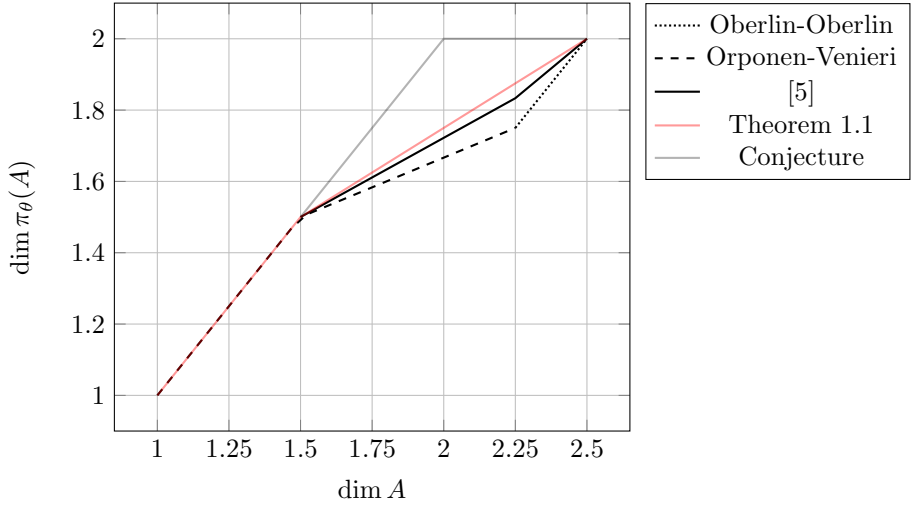

The improvement over \cite{harris20} in Theorem~\ref{mainthm} stems from Definition~\ref{alpha0defn}, the application of which is the main novelty of this work. Definition~\ref{alpha0defn} reformulates the projection problem as an averaged inequality over collections of ``bad'' tubes. The proof of Theorem~\ref{mainthm} follows a similar approach to \cite{harris20} (which used ideas from \cite{oberlin,venieri,GIOW,liu}), and proceeds by splitting an integral into a ``good'' and ``bad'' part. A self-similarity in the proof allows the ``bad'' part to be bounded by re-using Definition~\ref{alpha0defn}, which circumvents the appeal to the Orponen-Venieri lemma (\cite[Lemma 2.3]{venieri}). Since the use of this lemma was the only step in \cite{harris20} specific to non-great circles, this allows the proof to be generalised to curves in $S^2$ of nonvanishing geodesic curvature. The bound on the ``good'' part uses the decoupling theorem for the cone in $\mathbb{R}^3$, from \cite{bourgaindemeter}. 

Section~\ref{mainthmproof} contains the proof of Theorem~\ref{mainthm}. Section~\ref{rsi} is a proof of the refined Strichartz inequality, needed in Section~\ref{mainthmproof}. Section~\ref{wavepacket} contains a derivation of the wave packet decomposition needed in Section~\ref{mainthmproof}, and is independent of Section~\ref{rsi}. Section~\ref{discussion} is a discussion of some related problems. 

\subsection*{Acknowledgements} Alex Barron helped find the counterexample at the end of Section~\ref{discussion}. I also thank Shaoming Guo for some comments on an earlier version of this article.
\section{Refined Strichartz inequality} \label{rsi}

\begin{definition} Given a $C^3$ curve $\gamma: [a,b] \to S^2$ with $\det(\gamma, \gamma', \gamma'')$ nonvanishing, for each $R \geq 1$ let 
\[ \Theta_R = \left\{ \frac{j}{10^5B^{10} R^{1/2}}  : j \in \mathbb{Z} \right\} \cap [a,b], \]
where $B \geq 1$ is the smallest constant such that
\[ \left\lvert \det\left( \gamma, \gamma', \gamma'' \right) \right\rvert \geq B^{-1} , \]
and 
\[\left\lVert \gamma\right \rVert_{C^3[a,b]} \leq B. \]
For each $\theta \in \Theta_R$, let 
\begin{multline*} \tau(\theta) =\bigg\{ x_1 \gamma(\theta) + x_2 \frac{\gamma'(\theta )}{\left\lvert \gamma'(\theta ) \right\rvert} + x_3 \frac{(\gamma \times \gamma' )(\theta )}{\left\lvert (\gamma \times \gamma' )(\theta )\right\rvert} : \\
 1 \leq x_1 \leq 2, \, \left\lvert x_2\right\rvert \leq R^{-1/2}, \, \left\lvert x_3\right\rvert \leq R^{-1}  \bigg\}. \end{multline*}
Let
\[ \mathcal{P}_{R^{-1}} \left( \Gamma(\gamma) \right)  = \left\{ \tau(\theta): \theta\in \Theta_R \right\}. \]
Given $\delta>0$ and $\tau \in \mathcal{P}_{R^{-1} }\left( \Gamma(\gamma) \right)$, let 
\begin{multline*} T_{\tau,0} = \bigg\{ x_1\frac{\left(\gamma \times \gamma' \right)(\theta_{\tau} )}{\left\lvert \left(\gamma \times \gamma' \right)(\theta_{\tau} )\right\rvert}+ x_2 \frac{\gamma'(\theta_{\tau} )}{\left\lvert \gamma'(\theta_{\tau} ) \right\rvert} + x_3 \gamma(\theta_{\tau} ) : \\
\left\lvert x_1\right\rvert \leq R^{1+\delta}, \, \left\lvert x_2\right\rvert \leq R^{1/2+\delta}, \, \left\lvert x_3\right\rvert \leq R^{1/2+\delta}  \bigg\}. \end{multline*} \end{definition} 
	
 \begin{theorem} \label{refinedstrichartz} Let $\gamma: [a,b] \to S^2$ be a $C^3$ curve with $\det(\gamma, \gamma', \gamma'')$ nonvanishing. Let $B \geq 1$ be such that 
\begin{equation} \label{gammacdn1} \left\lvert \det\left( \gamma, \gamma', \gamma'' \right) \right\rvert \geq B^{-1} , \end{equation}
and 
\begin{equation} \label{gammacdn2} \left\lVert \gamma\right \rVert_{C^3[a,b]} \leq B. \end{equation}
Then for any $A \geq 1$ and $\epsilon>0$, there exists $\delta_0 >0$ such that the following holds for all $0 < \delta < \delta_0$. Let $R \geq 1$ and suppose that
\[ f = \sum_{T \in \mathbb{W}} f_T,  \]
where 
\[ \mathbb{W} \subseteq  \bigcup_{\tau \in \mathcal{P}_{R^{-1}}\left( \Gamma(\gamma ) \right)}  \mathbb{T}_{\tau}, \]
and each $\mathbb{T}_{\tau}$ is an $A$-overlapping set of translates of $T_{\tau,0}$ intersecting $B(0,R)$. Assume that for all $T \in \mathbb{W}$,
\begin{equation} \label{microlocalised}  \left\lVert f_T \right\rVert_{L^{\infty} \left( B(0,R) \setminus T \right) }+\sup_{q \in [1,2]}\left\lVert \widehat{f_T} \right\rVert_{L^q\left(\mathbb{R}^3 \setminus \tau(T) \right)} \leq A R^{-10000} \left\lVert f_T \right\rVert_2, \end{equation}
with $\left\lVert f_T\right\rVert_2$ constant over $T \in \mathbb{W}$ up to a factor of 2. Let $Y$ be a disjoint union of $R^{1/2}$-balls in $B(0,R)$, each of which intersects at most $M$ sets $2T$ with $T \in \mathbb{W}$. Then for $2 \leq p \leq 6$, 
\[ \left\lVert f \right\rVert_{L^p(Y)} \leq C_{\epsilon,\delta,A} B^{10^{10}/\epsilon} R^{\epsilon} \left( \frac{MR^{-3/2}}{\left\lvert \mathbb{W} \right\rvert } \right)^{\frac{1}{2} - \frac{1}{p} } \left( \sum_{T \in \mathbb{W}} \left\lVert f_T \right\rVert_2^2 \right)^{1/2}. \] \end{theorem}

\begin{proof} Assume that $[a,b] = [-1,1]$. Fix $A \geq 1$, $\epsilon \in (0,1/2)$, $\delta_0 = \epsilon^{100}$, $\delta \in (0, \delta_0)$, 
\[ R \geq \min\left\{B^{10^3/\epsilon}, 2^{10^5/\epsilon}\right\}, \]
and assume inductively that the theorem holds with $[a,b]=[-1,1]$ for all $\widetilde{R} \leq R^{3/4}$, for all curves $\gamma$ satisfying \eqref{gammacdn1} and \eqref{gammacdn2}, and for all $B \geq 1$. 

For each $\tau \in \mathcal{P}_{R^{-1}}(\Gamma(\gamma))$, let $\kappa = \kappa(\tau) \in \mathcal{P}_{R^{-1/2} }(\Gamma(\gamma))$ be the element of $\mathcal{P}_{R^{-1/2} }(\Gamma(\gamma))$ which minimises $\left\lvert \theta_{\tau} - \theta_{\kappa} \right\rvert$. For each $\kappa$, let 
\begin{multline*} \Box_{\kappa, 0} = \Big\{ x_1 \frac{\left(\gamma \times \gamma' \right)(\theta_{\kappa} )}{\left\lvert \left(\gamma \times \gamma' \right)(\theta_{\kappa} ) \right\rvert }+ x_2 \frac{\gamma'(\theta_{\kappa} )}{\left\lvert \gamma'(\theta_{\kappa} ) \right\rvert} + x_3 \gamma(\theta_{\kappa} ) : \\
\left\lvert x_1\right\rvert \leq R^{1+\delta}, \, \left\lvert x_2\right\rvert \leq R^{3/4+\delta}, \, \left\lvert x_3\right\rvert \leq R^{1/2+\delta}  \Big\}, \end{multline*}
and
\begin{multline*} \mathbb{P}_{\kappa} = \Big\{ \Box = a \gamma(\theta_{\kappa} )+  b \frac{\gamma'(\theta_{\kappa} )}{\left\lvert \gamma'(\theta_{\kappa} ) \right\rvert} +\Box_{\kappa, 0} :  \\
  a\in (1/10) R^{1/2+\delta} \mathbb{Z}, \, b \in (1/10) R^{3/4+\delta} \mathbb{Z} \Big\}.  \end{multline*} 
Let $\mathbb{P} = \bigcup_{\kappa \in \mathcal{P}_{R^{-1/2} }(\Gamma(\gamma))} \mathbb{P}_{\kappa}$. Given any $\tau$ and corresponding $\kappa = \kappa(\tau)$, 
\begin{equation} \label{angledistortion} \left\lvert \left\langle (\gamma \times \gamma')(\theta_{\tau}), \gamma'(\theta_{\kappa}) \right\rangle \right\rvert \leq B^{-7} R^{-1/4}, \end{equation}
and 
\begin{equation} \label{angledistortion2} \left\lvert \left\langle (\gamma \times \gamma')(\theta_{\tau}), \gamma(\theta_{\kappa}) \right\rangle \right\rvert \leq B^{-7}R^{-1/2}. \end{equation}
 It follows that for each $T \in \mathbb{T}_{\tau}$, there are $\sim 1$ sets $\Box \in \mathbb{P}_{\kappa(\tau)}$ with $T \cap 10\Box \neq \emptyset$, and moreover $T \subseteq 100\Box$ whenever $T \cap 10\Box \neq \emptyset$. For each such $T$ let $\Box = \Box(T)\in \mathbb{P}_{\kappa}$ be some choice such that $T \cap 10\Box \neq \emptyset$. For each $\tau$, let 
\begin{multline*} S_{\tau,0} = \bigg\{ x_1\frac{\left(\gamma \times \gamma' \right)(\theta_{\tau} )}{\left\lvert \left(\gamma \times \gamma' \right)(\theta_{\tau} )\right\rvert}+ x_2 \frac{\gamma'(\theta_{\tau} )}{\left\lvert \gamma'(\theta_{\tau} ) \right\rvert} + x_3 \gamma(\theta_{\tau} ) : \\
\left\lvert x_1\right\rvert \leq R^{1+\delta/2}, \, \left\lvert x_2\right\rvert \leq R^{1/2+\delta/2}, \, \left\lvert x_3\right\rvert \leq R^{1/4+\delta/2}  \bigg\}, \end{multline*}
and 
\begin{multline*} \mathbb{S}_{\tau} = \bigg\{ S = a \gamma(\theta_{\tau} )+  b \frac{\gamma'(\theta_{\tau} )}{\left\lvert \gamma'(\theta_{\tau} )\right\rvert} +S_{\tau,0}:  \\
   a\in (1/10) R^{1/4+\delta/2} \mathbb{Z}, \, b \in (1/10) R^{1/2+\delta/2} \mathbb{Z} \bigg\}. \end{multline*} 
For each $T \in \mathbb{W}$ and $S \in \mathbb{S}_{\tau(T)}$, let  
\[ f_S = f_{S,T} := \eta_S f_T, \]
where $\{\eta_S\}_{S \in \mathbb{S}_{\tau}}$ is a smooth partition of unity, such that 
\[ \sup_{q \in [1, \infty]} \left\lVert \eta_S\right\rVert_{L^q(100\Box \setminus 0.99S ) }  \lesssim R^{-10^6}, \]
where the implicit constant is absolute, with $\widehat{\eta_S}$ supported in 
\begin{multline*} \bigg\{ x_1 \gamma(\theta_{\tau} )+ x_2 \frac{\gamma'(\theta_{\tau} )}{\left\lvert \gamma'(\theta_{\tau} ) \right\rvert} + x_3 \frac{\left(\gamma \times \gamma' \right)(\theta_{\tau} )}{\left\lvert \left(\gamma \times \gamma' \right)(\theta_{\tau} )\right\rvert} : \\ 
\left\lvert x_1\right\rvert \leq 10^{-3}, \, \left\lvert x_2\right\rvert \leq 10^{-3}R^{-1/2}, \, \left\lvert x_3\right\rvert \leq 10^{-3}R^{-1}  \bigg\}. \end{multline*}
This partition can be constructed using the Poisson summation formula. By dyadic pigeonholing and by \eqref{microlocalised}, there are dyadic numbers $\mu$ and $\nu$ such that
\begin{equation} \label{cmu} \left\lVert f\right\rVert_{L^p(Y)} \lesssim \left(\log R\right)^2 \left\lVert \sum_{\Box \in \mathbb{P}} \sum_{(S,T) \in \mathbb{W}_{\Box}} f_{S,T} \right\rVert_{L^p(Y)}   + R^{-1000} \left( \sum_{T \in \mathbb{W} } \left\lVert f_T \right\rVert_2^2 \right)^{1/2}, \end{equation}
where, for each $\Box$, $\mathbb{W}_{\Box}$ is a subset of 
\[ \left\{ (S,T)  :  T \in \mathbb{W}, \quad \Box = \Box(T), \quad S \in \mathbb{S}_{\tau(T)  }, \quad  \left\lVert \eta_S f_T\right\rVert_2 \in [ \nu,2\nu)\right\}, \]
such that for any $(S_0, T_0) \in  \mathbb{W}_{\Box}$, 
\[ \left\lvert \left\{ (S,T_0)  \in \mathbb{W}_{\Box} \right\} \right\rvert \in [\mu, 2\mu ), \]
and with the property that 
\[ (S,T) \in \mathbb{W}_{\Box} \Rightarrow S \cap T \neq \emptyset. \]
The dyadic range of $\nu$ was constrained; relying on the tail term in \eqref{cmu} to handle the contribution from those $f_{S,T}$ with $\lVert f_{S,T}\rVert_2\leq R^{-2000} \lVert f_T \rVert_2$. Hence
\[ \left\lVert f_{S,T} \right\rVert_{L^{\infty} \left( \mathbb{R}^3 \setminus 0.99S \right) }+\sup_{q \in [1,2]}\left\lVert \widehat{f_{S,T}} \right\rVert_{L^q\left(\mathbb{R}^3 \setminus 1.1\tau(T) \right)} \lesssim AR^{-7000} \left\lVert f_{S,T} \right\rVert_2, \]
for all $\Box$ and $(S,T) \in \mathbb{W}_{\Box}$, where the implicit constant is absolute.  

For each $\kappa$ and $\Box \in \mathbb{P}_{\kappa}$, let $\left\{Q_{\Box} \right\}_{Q_{\Box}}$ be a finitely overlapping cover of $100\Box$ by translates of the ellipsoid
\begin{multline*} \Bigg\{ x_1 \gamma(\theta_{\kappa} ) + x_2 \frac{\gamma'(\theta_{\kappa} )}{\left\lvert \gamma'(\theta_{\kappa} )\right\rvert} + x_3 \frac{\left( \gamma \times \gamma' \right)(\theta_{\kappa} )}{\left\lvert \left( \gamma \times \gamma' \right)(\theta_{\kappa} )\right\rvert}: \\
 \left( \left\lvert x_1\right\rvert ^2 + \left(\left\lvert x_2\right\rvert R^{-1/4}\right)^2 + \left(\left\lvert x_3 \right\rvert R^{-1/2} \right)^2 \right)^{1/2} \leq  R^{1/4+\delta/2} \Bigg\}. \end{multline*} 
Using Poisson summation again, let $\{ \eta_{Q_{\Box}} \}_{Q_{\Box} \in \mathcal{Q}_{\Box}}$ be a smooth partition of unity such that on $10^3 \Box$,
\[ \sum_{Q_{\Box} \in \mathcal{Q}_{\Box}} \eta_{Q_{\Box}} =1, \]
and such that each $\eta_{Q_{\Box}}$ satisfies 
\[ \left\lVert \eta_{Q_{\Box}}\right\rVert_{\infty} \lesssim 1, \quad  \left\lVert \eta_{Q_{\Box}}\right\rVert_{L^{\infty}(\mathbb{R}^3 \setminus Q_{\Box} ) } \lesssim R^{-10000}, \]
and 
\[  \left\lvert \eta_{Q_{\Box} }(x) \right\rvert \lesssim  \dist(x, Q_{\Box})^{-10000} \quad \forall \, x \in \mathbb{R}^3,  \]
with $\widehat{\eta_{Q_{\Box}}}$ supported in 
\begin{multline*} \bigg\{ \xi_1 \gamma(\theta_{\kappa} ) + \xi_2 \frac{\gamma'(\theta_{\kappa} )}{\left\lvert \gamma'(\theta_{\kappa} )\right\rvert} + \xi_3 \frac{\left( \gamma \times \gamma' \right)(\theta_{\kappa} )}{\left\lvert \left( \gamma \times \gamma' \right)(\theta_{\kappa} )\right\rvert} \\
:  \left\lvert \xi_1\right\rvert \leq R^{-1/4}, \quad \left\lvert \xi_2\right\rvert \leq R^{-1/2}, \quad \left\lvert \xi_3\right\rvert \leq R^{-3/4} \bigg\}. \end{multline*}
  By dyadic pigeonholing,  
\begin{multline*} \left\lVert \sum_{\Box} \sum_{(S,T) \in \mathbb{W}_{\Box}} f_{S,T} \right\rVert_{L^p(Y)} \lesssim \log R  \left\lVert \sum_{\Box} \sum_{(S,T) \in \mathbb{W}_{\Box}} \eta_{Y_{\Box}} f_{S,T} \right\rVert_{L^p(Y)} \\
+ AR^{-1000} \left( \sum_{T \in \mathbb{W} } \left\lVert f_T \right\rVert_2^2 \right)^{1/2}, \end{multline*}
where, for each $\Box$, $Y_{\Box}$ is a union over a subset of the sets $Q_{\Box}$, and $\eta_{Y_{\Box}}$ is the corresponding sum over $\eta_{Q_{\Box}}$, such that each $Q_{\Box} \subseteq Y_{\Box}$ intersects a number $\# \in [M'(\Box), 2M'(\Box))$ different sets $3S$ with $(S,T) \in \mathbb{W}_{\Box}$, up to a factor of 2. By pigeonholing again,
\[ \left\lVert \sum_{\Box} \sum_{(S,T) \in \mathbb{W}_{\Box}} \eta_{Y_{\Box}} f_{S,T} \right\rVert_{L^p(Y)} \lesssim (\log R )^2 \left\lVert \sum_{\Box \in \mathbb{B}} \sum_{(S,T) \in \mathbb{W}_{\Box}} \eta_{Y_{\Box}} f_{S,T} \right\rVert_{L^p(Y)}, \]
where $\left\lvert \mathbb{W}_{\Box} \right\rvert$ and $M'=M'(\Box)$ are constant over $\Box \in \mathbb{B}$ up to a factor of 2. By one final pigeonholing step,
\[ \left\lVert \sum_{\Box \in \mathbb{B}} \sum_{(S,T) \in \mathbb{W}_{\Box}} \eta_{Y_{\Box}} f_{S,T} \right\rVert_{L^p(Y)} \lesssim \log R \left\lVert \sum_{\Box \in \mathbb{B}} \sum_{(S,T) \in \mathbb{W}_{\Box}} \eta_{Y_{\Box}} f_{S,T} \right\rVert_{L^p(Y')}, \] 
where $Y'$ is a union over $R^{1/2}$-balls $Q \subseteq Y$ such that each ball $2Q$ intersects a number $\# \in [M'', 2M'')$ of the sets $Y_{\Box}$ in a set of strictly positive Lebesgue measure, as $\Box$ varies over $\mathbb{B}$. Fix $Q \subseteq Y'$. By the decoupling theorem for generalised $C^3$ cones (see \cite[Exercise 12.5]{demeter}), followed by Hölder's inequality,
\begin{multline*} \left\lVert \sum_{\Box \in \mathbb{B}} \sum_{(S,T) \in \mathbb{W}_{\Box}} \eta_{Y_{\Box}} f_{S,T} \right\rVert_{L^p(Q)} \\
\leq C_{\epsilon} B^{100} R^{\epsilon/100} \left( M'' \right)^{\frac{1}{2} - \frac{1}{p}} \left(  \sum_{\Box \in \mathbb{B}} \left\lVert\sum_{(S,T) \in \mathbb{W}_{\Box}} \eta_{Y_{\Box}} f_{S,T} \right\rVert_{L^p(2Q)}^p \right)^{1/p} \\
+ R^{-900} \left( \sum_{T \in \mathbb{W}} \left\lVert f_T \right\rVert_2^2 \right)^{1/2}. \end{multline*}
Summing over $Q$ gives 
\begin{multline*} \left\lVert f\right\rVert_{L^p(Y)} \lesssim C_{\epsilon}\left( \log R \right)^{100} B^{100} R^{\epsilon/100} \left( M'' \right)^{\frac{1}{2} - \frac{1}{p}} \\
\times \left(  \sum_{\Box \in \mathbb{B}} \left\lVert\sum_{(S,T) \in \mathbb{W}_{\Box}} f_{S,T} \right\rVert_{L^p(Y_{\Box} ) }^p \right)^{1/p} + AR^{-800} \left( \sum_{T \in \mathbb{W}} \left\lVert f_T \right\rVert_2^2 \right)^{1/2}. \end{multline*} This will be bounded using the inductive assumption, following a Lorentz rescaling.

For each $\theta \in [-1,1]$, define the Lorentz rescaling map $L = L_{\theta}$ at $\theta$ by
\begin{multline*} L\left[x_1 \gamma(\theta) + x_2 \frac{\gamma'(\theta)}{\left\lvert \gamma'(\theta)\right\rvert}  + x_3 \frac{\left(\gamma \times \gamma' \right)(\theta)}{\left\lvert \left(\gamma \times \gamma' \right)(\theta)\right\rvert}\right] \\
=x_1 \gamma(\theta) + R^{1/4} x_2  \frac{\gamma'(\theta)}{\left\lvert \gamma'(\theta)\right\rvert}  + R^{1/2} x_3 \frac{\left(\gamma \times \gamma' \right)(\theta)}{\left\lvert \left(\gamma \times \gamma' \right)(\theta)\right\rvert}. \end{multline*}
Let 
\[ \widetilde{\gamma}(\phi) = \frac{ L(\gamma(\phi))}{\left\lvert L(\gamma(\phi) ) \right\rvert}, \qquad \phi \in [-1,1]. \]
Then for any $\phi \in [-1,1]$,
\[ \widetilde{\gamma}'(\phi) = \frac{\pi_{\widetilde{\gamma}(\phi)^{\perp} }\left( L(\gamma'(\phi) ) \right) }{\left\lvert L(\gamma(\phi) ) \right\rvert}, \]
and 
\[ \widetilde{\gamma}''(\phi) = \frac{\pi_{\widetilde{\gamma}(\phi)^{\perp} }\left( L(\gamma''(\phi) ) \right) }{\left\lvert L(\gamma(\phi) ) \right\rvert} - \frac{\left\langle L(\gamma(\phi)), L(\gamma'(\phi) ) \right\rangle \pi_{\widetilde{\gamma}(\phi)^{\perp} }\left( L(\gamma'(\phi) ) \right) }{\left\lvert L(\gamma(\phi) ) \right\rvert^3}. \]
Hence 
\begin{align*} \det( \widetilde{\gamma}, \widetilde{\gamma}',\widetilde{\gamma}'' ) &= \frac{1}{\left\lvert L \circ \gamma \right\rvert^3}  \det\left( L \circ \gamma, L \circ \gamma', L \circ \gamma''\right) \\
&= \frac{R^{3/4}}{\left\lvert L \circ \gamma \right\rvert^3}  \det\left( \gamma, \gamma', \gamma''\right). \end{align*}
Let $\varepsilon = (10^5 B^{10})^{-1}$, and for fixed $\theta \in [-1+\varepsilon,1-\varepsilon]$, let
\[ \widehat{\gamma}(\phi) = \widetilde{\gamma}(\theta +   R^{-1/4}\phi  ), \quad  \phi \in [-\varepsilon, \varepsilon]. \]
The assumption that $\left\lVert \gamma\right \rVert_{C^3[-1,1]} \leq B$ yields
\[ 1 \leq \left\lvert L(\gamma(\phi) ) \right\rvert \leq 1+ 10B \varepsilon, \quad \forall \, \phi \in [\theta - \varepsilon R^{-1/4}, \theta + \varepsilon R^{-1/4} ]. \]
Similarly, 
\[ \left\lvert L(\gamma'(\phi)) - L(\gamma'(\theta)) \right\rvert \leq 10\varepsilon B R^{1/4},  \quad \forall \, \phi \in [\theta - \varepsilon R^{-1/4}, \theta + \varepsilon R^{-1/4} ], \]
\[ \left\lvert L(\gamma''(\phi)) - L(\gamma''(\theta) ) \right\rvert \leq 10\varepsilon B R^{1/4},  \quad \forall \, \phi \in [\theta - \varepsilon R^{-1/4}, \theta + \varepsilon R^{-1/4} ],  \]
and 
\[\left\lvert L \gamma'''(\phi) \right\rvert \leq BR^{1/2}, \quad \forall \, \phi \in [\theta - \varepsilon R^{-1/4}, \theta + \varepsilon R^{-1/4} ]. \]
It follows that
\[ \left\lvert \det\left(\widehat{\gamma}, \widehat{\gamma}', \widehat{\gamma}''\right) \right\rvert \geq (2B)^{-1} \]
on $[-\varepsilon,\varepsilon]$, and that
\[ \left\lVert \widehat{\gamma} \right\rVert_{C^3[-\varepsilon,\varepsilon] } \leq 2B, \]
(the calculation for the third derivative is omitted).

For each $\Box \in \mathbb{B}$, given $(S,T) \in \mathbb{W}_{\Box}$, let $g_{S,T} =  f_{S,T} \circ L$, where $L= L_{\theta_{\kappa(\Box) } }$. Then
\begin{equation} \label{inductthis} \left\lVert\sum_{(S,T) \in \mathbb{W}_{\Box}} f_{S,T} \right\rVert_{L^p(Y_{\Box} ) } \leq R^{\frac{3}{4p}}\left\lVert\sum_{(S,T) \in \mathbb{W}_{\Box}} g_{S,T} \right\rVert_{L^p(L^{-1} Y_{\Box} ) }. \end{equation}
The inequalities \eqref{angledistortion} and \eqref{angledistortion2} imply that for each $(S,T) \in \mathbb{W}_{\Box}$, the set $L^{-1}(S)$ is a equivalent (up to a factor 1.01) to a box of length $R^{1/2+\delta/2}$ in its longest direction parallel to $L^{-1}(\gamma \times \gamma')(\theta_{\tau})$ and of length $R^{1/4+\delta/2}$ in its other two directions. The ellipsoids $Q_{\Box}$ are rescaled to $R^{1/4+\delta/2}$-balls $L^{-1}(Q_{\Box})$. Moreover, it will be shown that
\begin{multline} \label{scaledtau} L(\tau) \subseteq \bigg\{ x_1 \widetilde{\gamma}(\theta_{\tau} ) + x_2 \frac{\widetilde{\gamma}'(\theta_{\tau})}{\left\lvert \widetilde{\gamma}'(\theta_{\tau}) \right\rvert}  + x_3 \frac{ \left( \widetilde{\gamma}  \times \widetilde{\gamma}'\right)(\theta_{\tau})}{\left\lvert \left( \widetilde{\gamma}  \times \widetilde{\gamma}'\right)(\theta_{\tau})\right\rvert } \\
: 1 \leq x_1 \leq 2.01, \, \left\lvert x_2\right\rvert \leq (1.01)R^{-1/4}, \, \left\lvert x_3\right\rvert \leq R^{-1/2} \bigg\}. \end{multline} 
To prove this, let
\[ x =  x_1 \gamma(\theta_{\tau}) + x_2 \frac{\gamma'(\theta_{\tau})}{\left\lvert\gamma'(\theta_{\tau}) \right\rvert} + x_3 \frac{(\gamma \times \gamma')(\theta_{\tau})}{\left\lvert (\gamma \times \gamma')(\theta_{\tau})\right\rvert} \in \tau, \]
where 
\[ x_1 \in [1,2], \quad \left\lvert x_2\right\rvert \leq R^{-1/2}, \quad \left\lvert x_3\right\rvert \leq R^{-1}. \]
The vector $\left(\widetilde{\gamma} \times \widetilde{\gamma}'\right)(\theta_{\tau})$ is parallel to $L^{-1}((\gamma \times \gamma')(\theta_{\tau}))$, since $L^{-1}((\gamma \times \gamma')(\theta_{\tau}))$ is orthogonal to $\widetilde{\gamma}(\theta_{\tau})$ and $\widetilde{\gamma}'(\theta_{\tau})$. The inequality
\[ \left\lvert L^{-1}((\gamma \times \gamma')(\theta_{\tau})) \right\rvert \geq R^{-1/2}\left\lvert  (\gamma \times \gamma')(\theta_{\tau}) \right\rvert \]
gives
\begin{equation} \label{component1} \left\lvert \left\langle Lx, \frac{ L^{-1}((\gamma \times \gamma')(\theta_{\tau}) )}{\left\lvert  L^{-1}((\gamma \times \gamma')(\theta_{\tau}) ) \right\rvert }\right\rangle \right\rvert \leq R^{-1/2}. \end{equation}
Moreover, 
\begin{align} \notag &\left\lvert \left\langle Lx, \frac{ \pi_{L(\gamma(\theta_{\tau}))^{\perp}} \left( L\left(\gamma'(\theta_{\tau}) \right)\right)}{\left\lvert \pi_{L(\gamma(\theta_{\tau}))^{\perp}} \left( L\left(\gamma'(\theta_{\tau}) \right)\right) \right\rvert} \right\rangle \right\rvert \\
\notag & \quad = \left\lvert \left\langle x_2 L\left(\frac{\gamma'(\theta_{\tau})}{\left\lvert\gamma'(\theta_{\tau}) \right\rvert}\right) + x_3 L\left(\frac{(\gamma \times \gamma')(\theta_{\tau})}{\left\lvert (\gamma \times \gamma')(\theta_{\tau})\right\rvert}\right), \frac{ \pi_{L(\gamma(\theta_{\tau}))^{\perp}} \left( L\left(\gamma'(\theta_{\tau}) \right)\right)}{\left\lvert \pi_{L(\gamma(\theta_{\tau}))^{\perp}} \left( L\left(\gamma'(\theta_{\tau}) \right)\right) \right\rvert} \right\rangle \right\rvert \\
\label{component2} &\quad \leq (1.01)R^{-1/4}.
 \end{align} 
For the direction $L(\gamma(\theta_{\tau}))$, 
\begin{equation} \label{component3} \left\langle Lx, \frac{ L(\gamma(\theta_{\tau}) )}{\left\lvert L(\gamma(\theta_{\tau}) ) \right\rvert } \right\rangle = x_1\left\lvert L(\gamma(\theta_{\tau} ) ) \right\rvert + O(R^{-1/4}). \end{equation}
Combining \eqref{component1}, \eqref{component2} and \eqref{component3} gives \eqref{scaledtau}. 

Inductively applying the theorem at scale $R^{1/2}$ therefore gives
\begin{align*} \eqref{inductthis} &\lesssim C_{\epsilon,\delta,A} B^{10^{10}/\epsilon}R^{3\epsilon/4} R^{ 3/(4p)} \left( \frac{ M' R^{-3/4} }{\left\lvert \mathbb{W}_{\Box} \right\rvert } \right)^{\frac{1}{2} - \frac{1}{p} } \left( \sum_{(S,T) \in \mathbb{W}_{\Box} } \left\lVert g_{S,T} \right\rVert_2^2 \right)^{1/2} \\
&= C_{\epsilon,\delta,A} B^{10^{10}/\epsilon}R^{3\epsilon/4} \left( \frac{ M' R^{-3/2} }{\left\lvert \mathbb{W}_{\Box} \right\rvert } \right)^{\frac{1}{2} - \frac{1}{p} } \left( \sum_{(S,T) \in \mathbb{W}_{\Box} } \left\lVert f_{S,T} \right\rVert_2^2 \right)^{1/2}, \end{align*} 
for each $\Box \in \mathbb{B}$. Hence 
\begin{multline*} \left\lVert f\right\rVert_{L^p(Y)}  \\
\leq C_{\epsilon,\delta,A} B^{10^{10}/\epsilon} R^{4\epsilon/5} \left( \frac{ M'M'' R^{-3/2} }{\left\lvert \mathbb{W}_{\Box} \right\rvert } \right)^{\frac{1}{2} - \frac{1}{p} } \left(\sum_{\Box \in \mathbb{B}} \left( \sum_{(S,T) \in \mathbb{W}_{\Box} } \left\lVert f_{S,T} \right\rVert_2^2 \right)^{p/2}\right)^{1/p}. \end{multline*} 
Using $\left\lVert f_{S,T}\right\rVert_2^2 \lesssim \left\lVert f_T\right\rVert_2^2/\mu$, this is 
\[ \lesssim C_{\epsilon,\delta,A} B^{10^{10}/\epsilon} R^{4\epsilon/5} \left( \frac{ M'M'' R^{-3/2} }{\left\lvert \mathbb{W} \right\rvert \mu } \right)^{\frac{1}{2} - \frac{1}{p} } \left( \frac{ \left\lvert \mathbb{B} \right\rvert \left\lvert \mathbb{W}_{\Box} \right\rvert }{\left\lvert \mathbb{W}\right\rvert \mu } \right)^{\frac{1}{p} }\left( \sum_{T \in \mathbb{W} } \left\lVert f_T \right\rVert_2^2 \right)^{1/2}. \]
The second bracketed term is $\lesssim 1$, since
\[ \left\lvert \mathbb{W}  \right\rvert \mu = \sum_{T \in \mathbb{W}} \mu  \geq \sum_{\Box \in \mathbb{B}} \sum_{\substack{ T \in \mathbb{W}: \\
\Box = \Box(T)}} \mu \sim \sum_{\Box \in \mathbb{B}} \sum_{\substack{ T \in \mathbb{W}: \\
\Box = \Box(T)}} \sum_{ \substack{ (S,T') \in \mathbb{W}_{\Box}: \\ 
T'=T } } 1 = \left\lvert \mathbb{B} \right\rvert \left\lvert \mathbb{W}_{\Box} \right\rvert. \]
It remains to show that $M'M'' \lesssim \mu M$. Let $Q \subseteq Y'$ be any $R^{1/2}$-ball. By definition of $\mu$ and $M$,
\begin{align*} \mu M &\gtrsim \sum_{\substack{ T \in \mathbb{W}:  \\
2T \cap Q \neq \emptyset }}  \sum_{\substack{\Box \in \mathbb{B}: \\ 
\Box = \Box(T)}}  \sum_{\substack{(S,T') \in \mathbb{W}_{\Box} \\
T' = T }}1  \\
&= \sum_{\Box \in \mathbb{B}} \sum_{\substack{ T\in \mathbb{W}: \\
\Box = \Box(T) \\
2T \cap Q \neq \emptyset}}   \sum_{\substack{(S,T') \in \mathbb{W}_{\Box}: \\
T' = T}}1 \\
&= \sum_{\Box \in \mathbb{B}} \sum_{\substack{(S,T) \in \mathbb{W}_{\Box}: \\
2T \cap Q \neq \emptyset }} 1. \end{align*}  
By definition of $M'$ and $M''$, \begingroup
\allowdisplaybreaks
\begin{align} \notag M'M'' &\sim \sum_{\substack{\Box \in \mathbb{B}: \\ m(Y_{\Box} \cap 2Q)>0 }} M' \\
\notag &\leq  \sum_{\substack{\Box \in \mathbb{B}: \\ m(Y_{\Box} \cap 2Q)>0 }} \sum_{Q_{\Box} \subseteq Y_{\Box}} M' \frac{m(Q_{\Box} \cap 2Q)}{m(Y_{\Box} \cap 2Q) } \\ 
\notag &\sim  \sum_{\substack{\Box \in \mathbb{B}: \\ m(Y_{\Box} \cap 2Q)>0 }} \sum_{Q_{\Box} \subseteq Y_{\Box}} \sum_{\substack{(S,T) \in \mathbb{W}_{\Box}: \\
 Q_{\Box} \cap 3S \neq \emptyset }}  \frac{m(Q_{\Box} \cap 2Q)}{m(Y_{\Box} \cap 2Q) } \\
\notag &=  \sum_{\substack{\Box \in \mathbb{B}: \\ m(Y_{\Box} \cap 2Q)>0 }} \sum_{(S,T) \in \mathbb{W}_{\Box}} \sum_{\substack{Q_{\Box} \subseteq Y_{\Box}: \\ Q_{\Box} \cap 3S \neq \emptyset} }  \frac{m(Q_{\Box} \cap 2Q)}{m(Y_{\Box} \cap 2Q) } \\
\label{nontrivial} &\leq  \sum_{\substack{\Box \in \mathbb{B}: \\ m(Y_{\Box} \cap 2Q)>0 }} \sum_{ \substack{(S,T) \in \mathbb{W}_{\Box}: \\ 2T \cap Q \neq \emptyset }} \sum_{Q_{\Box} \subseteq Y_{\Box}}  \frac{m(Q_{\Box} \cap 2Q)}{m(Y_{\Box} \cap 2Q) } \\
\notag &\lesssim  \sum_{\Box \in \mathbb{B} } \sum_{ \substack{(S,T) \in \mathbb{W}_{\Box}: \\ 2T \cap Q \neq \emptyset  }} 1 \\
\notag &\lesssim \mu M.  \end{align} \endgroup 
The inequality \eqref{nontrivial} above follows from the observation that if $Q_{\Box} \cap 2Q \neq \emptyset$, and if $(S,T) \in \mathbb{W}_{\Box}$ is such that $Q_{\Box} \cap 3S \neq \emptyset$, then $2T \cap Q \neq \emptyset$. To prove this, recall that $L^{-1}(S)$ is equivalent (up to a factor 1.01) to a box of length $R^{1/2+\delta/2}$ in its longest direction and of length $R^{1/4+\delta/2}$ in its other two directions. The sets $L^{-1}(Q_{\Box})$ are $R^{1/4+\delta/2}$-balls. Therefore $Q_{\Box} \cap 3S \neq \emptyset$ implies that $Q_{\Box} \subseteq 100S \subseteq 1.5T$, which yields that $2T \cap Q \neq \emptyset$ since $Q_{\Box} \cap 2Q \neq \emptyset$. 
\end{proof} 

\section{Wave packet decomposition} \label{wavepacket} 
Throughout this section, assume that $\gamma: [a,b] \to S^2$ is $C^2$ with unit speed, that 
\begin{equation} \label{angleassumption} \left\lvert \gamma(\theta) - \gamma(\phi) \right\rvert \leq 2^{-100} \quad \text{and} \quad \left\lvert \gamma'(\theta) - \gamma'(\phi) \right\rvert \leq 2^{-100},  \end{equation}
for all $\theta, \phi \in [a,b]$, that 
\begin{equation} \label{meanvalue} \gamma_1' \neq 0, \quad \text{ where } \gamma = (\gamma_1, \gamma_2, \gamma_3), \end{equation}
and that 
\begin{equation} \label{pgc} \det(\gamma, \gamma', \gamma'' ) >0. \end{equation}
To simplify notation, the convention $\mathbb{N} = \{0,1,2, \dotsc \}$ will be assumed.   
\begin{definition} Let $\gamma: [a,b] \to S^2$ be a $C^2$ unit speed curve satisfying \eqref{angleassumption}, \eqref{meanvalue} and \eqref{pgc}. Let $\varrho>0$ and $\varepsilon \in (0,1]$. For each $k \geq 0$, let 
\begin{equation} \label{constantchoice} \Theta_k= \left\{a+\varrho 2^{-k/2}l : l \in \mathbb{N} \cap \left[0, (b-a) \cdot \varrho^{-1}2^{k/2}  \right) \right\}. \end{equation}
 For each $k \in [0, j ] \cap \mathbb{N}$, if $k < j$ then for each $\theta \in \Theta_k$, let
\begin{multline} \label{taudefn} \tau^+(\theta,j,k) = \Big\{ \lambda_1 \left(\gamma \times \gamma'\right)(\theta) + \lambda_2 \gamma'(\theta) + \lambda_3 \gamma(\theta) :\\
 2^{j-2} \leq \lambda_1 \leq 2^{j+2}, \quad \left\lvert \lambda_2\right\rvert \leq \varepsilon 2^{-k/2 + j}, \quad -2^{-k+j+2} \leq \lambda_3 \leq -2^{-k + j -2} \Big\}, \end{multline}
and
\begin{multline} \label{taudefn2} \tau^-(\theta,j,k) = \Big\{ \lambda_1 \left(\gamma \times \gamma'\right)(\theta) + \lambda_2 \gamma'(\theta) + \lambda_3 \gamma(\theta) \\
: -2^{j+2} \leq \lambda_1 \leq -2^{j-2}, \quad \left\lvert \lambda_2\right\rvert \leq \varepsilon 2^{-k/2 + j}, \quad 2^{-k + j -2} \leq  \lambda_3 \leq 2^{-k+j+2} \Big\}. \end{multline}
If $k=j$ then for each $\theta \in \Theta_j$, let
\begin{multline} \label{taudefn3} \tau^+(\theta,j,j) = \Big\{ \lambda_1 \left(\gamma \times \gamma'\right)(\theta) + \lambda_2 \gamma'(\theta) + \lambda_3 \gamma (\theta) :\\
 2^{j-2} \leq \lambda_1 \leq 2^{j+2}, \quad \left\lvert \lambda_2\right\rvert \leq \varepsilon 2^{j/2}, \quad \left\lvert \lambda_3\right\rvert \leq 4  \Big\}, \end{multline}
and
\begin{multline} \label{taudefn4}  \tau^-(\theta,j,j) = \Big\{ \lambda_1 \left(\gamma \times \gamma'\right)(\theta) + \lambda_2 \gamma'(\theta) + \lambda_3 \gamma(\theta) \\
: -2^{j+2} \leq \lambda_1 \leq -2^{j-2}, \quad  \left\lvert \lambda_2\right\rvert \leq \varepsilon 2^{j/2}, \quad \left\lvert \lambda_3\right\rvert \leq 4 \Big\}. \end{multline}
Let 
\[ \Lambda_{j,k}^+ = \left\{ \tau^+(\theta,j,k) : \theta \in \Theta_k\right\}, \quad  \Lambda_{j,k}^- = \left\{ \tau^-(\theta,j,k) : \theta \in \Theta_k\right\},\]
\[ \Lambda_{j,k} = \Lambda_{j,k}^+ \cup \Lambda_{j,k}^-, \]
and
\[ \Lambda = \bigcup_{j =0}^{\infty} \bigcup_{k=0}^j \Lambda_{j,k}. \] 
Given $J \geq 0$ let 
\[ \Lambda^J = \bigcup_{j \geq J} \bigcup_{J \leq k \leq j} \Lambda_{j,k}, \]
and
\[ \Lambda^{J,o} =\bigcup_{j > 2J} \bigcup_{J < k \leq j} \left\{ \tau \in \Lambda_{j,k} : \dist\left(\theta_{\tau} : \{a,b\}\right) > 100m^{-1}\max\{ \varepsilon, \varrho\} 2^{-k/2}  \right\}, \]
where $m$ is the minimum of $\left\lvert \det(\gamma, \gamma',\gamma'')\right\rvert$ on $[a,b]$. Given $\delta>0$, for each $j,k$ and each $\tau \in \Lambda_{j,k}$ let 
\begin{multline*} \mathbb{T}_{\tau} =\mathbb{T}_{\tau}^{\delta} = \Big\{ T = a_1\gamma(\theta_{\tau}) +a_2\gamma'(\theta_{\tau}) + a_3 \left( \gamma \times \gamma'\right)(\theta_{\tau})  + \\
\Big\{ \lambda_1 \gamma(\theta_{\tau}) + \lambda_2 \gamma'(\theta_{\tau}) + \lambda_3 \left(\gamma \times \gamma'\right)(\theta_{\tau}) : 
 \left\lvert \lambda_2\right\rvert, \left\lvert \lambda_3\right\rvert \leq 2^{ k/2-j + k\delta}, \quad \left\lvert \lambda_1 \right\rvert \leq 2^{k-j+k\delta } \Big\} \\
(a_2,a_3) \in 2^{-10 -j+k/2 + k\delta} \mathbb{Z}^2, \quad a_1 \in 2^{-10 + k-j+k\delta }\mathbb{Z} \Big\}. \end{multline*}
 \end{definition}
The parameter $J$ should be thought of as morally equal to 1, and is only used to exclude low frequency pieces. For any $\tau \in \Lambda$ and constant $C \geq 0$, let $C\tau$ be the box with the same centre as $\tau$ but with side lengths scaled by $C$.
\begin{lemma} \label{finiteoverlap} Let $\gamma: [a,b] \to S^2$ be a $C^2$ unit speed curve satisfying \eqref{angleassumption}, \eqref{meanvalue} and \eqref{pgc}. Let $\varepsilon \in (0,1]$ and $\varrho>0$. There exists a constant $C_{\gamma}$, such that if $\tau_1 \in \Lambda_{j_1,k_1}$, $\tau_2 \in \Lambda_{j_2,k_2}$, and 
\begin{equation} \label{intersection2} 1.01\tau_1 \cap 1.01\tau_2 \neq \emptyset, \end{equation}
then 
\[ \left\lvert j_1-j_2\right\rvert + \left\lvert k_1-k_2\right\rvert \leq C. \]
and 
\[ \left\lvert \theta_{\tau_1 } - \theta_{\tau_2 } \right\rvert \leq C 2^{-k_1/2}, \]
and such that if $\tau_1 \in \Lambda_{j_1,k_1}^+$ and  $\tau_2 \in \Lambda_{j_2,k_2}^-$, or if $\tau_1 \in \Lambda_{j_1,k_1}^-$ and  $\tau_2 \in \Lambda_{j_2,k_2}^+$, then 
\[ k_1+k_2 \leq C. \]
Moreover, there exists a constant $K$, depending only on $\gamma$ and $\varepsilon$, such that if $k_1,k_2 \geq K$ and \eqref{intersection2} holds, then
\begin{equation} \label{betteranglebound} \left\lvert \theta_{\tau_1 } - \theta_{\tau_2 } \right\rvert \leq C \varepsilon 2^{-k_1/2}. \end{equation}
\end{lemma}
\begin{proof}  If $1.01\tau_1 \cap 1.01\tau_2 \neq \emptyset$, let 
\begin{multline} \label{intersection} \lambda_1 \left(\gamma \times \gamma' \right)\left(\theta_{\tau_1} \right) + \lambda_2 \gamma'\left(\theta_{\tau_1} \right) + \lambda_3 \gamma\left(\theta_{\tau_1} \right)\\
 =  \mu_1\left(\gamma \times \gamma' \right)\left(\theta_{\tau_2} \right) + \mu_2 \gamma'\left(\theta_{\tau_2} \right) + \mu_3 \gamma\left(\theta_{\tau_2} \right), \end{multline}
be a point in the intersection, where each side satisfies the conditions in any of \eqref{taudefn},\eqref{taudefn2},\eqref{taudefn3} or \eqref{taudefn4}, multiplied by the factor 1.01. Then 
\begin{equation} \label{jcdn} \left\lvert j_1-j_2\right\rvert \leq 5, \end{equation}
by comparing norms on either side. By symmetry it may be assumed that $k_1 \leq k_2$. If $\sgn \lambda_1 \neq \sgn \mu_1$, then by \eqref{angleassumption}, 
\[ \left\lvert \lambda_1 \left(\gamma \times \gamma' \right)\left(\theta_{\tau_1} \right) + \lambda_2 \gamma'\left(\theta_{\tau_1} \right) + \lambda_3 \gamma\left(\theta_{\tau_1} \right) -  \mu_1 \left(\gamma \times \gamma' \right)\left(\theta_{\tau_2} \right) \right\rvert \geq 2^{j_1-3}, \]
and therefore $k_2 \leq 100$ by \eqref{intersection} and the triangle inequality. This shows that all conclusions of the lemma hold if $\sgn \lambda_1 \neq \sgn \mu_1$ (assuming that $K \geq 1000$ and $C \geq 2^{100} \max\{1, b-a\}$), so assume that $\sgn \lambda_1 = \sgn \mu_1$. By \eqref{intersection} and \eqref{jcdn}, 
\begin{equation} \label{scaledbound} \left\lvert \lambda_1\left(\gamma \times \gamma' \right)\left(\theta_{\tau_1} \right) -  \mu_1\left(\gamma \times \gamma' \right)\left(\theta_{\tau_2} \right) \right\rvert \leq  2^{j_1-k_1/2 + 8}. \end{equation}
The inequality 
\[ \left\lvert v-\lambda w \right\rvert \geq \left\lvert v-w\right\rvert/2 \quad   \forall \, \lambda \in [0,1] \quad \forall \, v,w \in S^2 \text{ with } \langle v,w \rangle \geq 0, \]
together with \eqref{angleassumption}, \eqref{scaledbound} and the assumption $\sgn \lambda_1 = \sgn \mu_1$, gives 
\begin{equation} \label{bothcases} \left\lvert \left(\gamma \times \gamma' \right)\left(\theta_{\tau_1} \right) -   \left(\gamma \times \gamma' \right)\left(\theta_{\tau_2} \right) \right\rvert \leq 2^{-k_1/2 + 15}. \end{equation}
 By \eqref{meanvalue}, \eqref{bothcases}, the mean value theorem and the identity
\[ \left( \gamma \times \gamma' \right)' = -\det\left( \gamma, \gamma', \gamma'' \right) \gamma', \]
 it follows that
\begin{equation} \label{anglebound2} \left\lvert \theta_{\tau_1} - \theta_{\tau_2} \right\rvert \leq \frac{2^{-k_1/2+15}}{ \min \left\lvert\gamma_1'\det(\gamma, \gamma', \gamma'') \right\rvert}. \end{equation}
A similar argument gives that 
\[ \left\lvert \theta_{\tau_1} - \theta_{\tau_2} \right\rvert \lesssim \varepsilon 2^{-k_1/2}, \]
provided $k_1$ and $k_2$ are sufficiently large depending on $\varepsilon$ and $\gamma$. 

If $k_1 =j_1$ then the lemma follows, so assume that $k_1 < j_1$. By the generalised mean value theorem and the assumption that $\gamma$ is $C^2$, 
\begin{equation} \label{generalisedmvt} \gamma(\theta) = \gamma\left(\phi \right) + \left(\theta-\phi \right) \gamma'\left(\phi \right) + \frac{1}{2} \left( \theta- \phi \right)^2\gamma''\left(\phi \right) + o\left( \left\lvert\theta - \phi\right\rvert^2 \right) \end{equation}
for any $\theta, \phi \in [a,b]$, where the rate of decay to zero in the error term is uniform in $\phi$ and $\theta$. Hence
\[ \left\langle \gamma(\theta), \left( \gamma \times \gamma'\right)\left(\phi \right) \right\rangle = \frac{1}{2} \left( \theta - \phi \right)^2 \det\left(\gamma\left(\phi \right), \gamma'\left(\phi \right), \gamma''\left(\phi \right) \right)  + o\left( \left\lvert\theta - \phi\right\rvert^2 \right). \]
Using \eqref{intersection}, \eqref{anglebound2}, letting $\theta =\theta_{\tau_1}$ and $\phi = \theta_{\tau_2}$, and taking the dot product of both sides of \eqref{intersection} with $\gamma\left(\theta_{\tau_1}\right)$, gives
\begin{multline*} \left\lvert \lambda_3 -\frac{\mu_1}{2} \left( \theta_{\tau_1} - \theta_{\tau_2} \right)^2\det\left( \gamma\left(\theta_{\tau_2}\right), \gamma'\left(\theta_{\tau_2} \right), \gamma''\left(\theta_{\tau_2} \right) \right) -\mu_1 o\left( \left\lvert \theta_{\tau_1} - \theta_{\tau_2} \right\rvert^2 \right) \right\rvert  \\
\leq C_{\gamma}2^{j_2-k_1/2- k_2/2}, \end{multline*}
for some constant $C_{\gamma}$ which may depend on $\gamma$. Since $\left\lvert \lambda_3\right\rvert \gtrsim \left\lvert 2^{j_1-k_1} \right\rvert$ and $\sgn \lambda_3 = -\sgn \lambda_1 = - \sgn \mu_1$, and since the decay to zero in the error term is uniform, the left-hand side is $\gtrsim 2^{j_1-k_1}$ provided $k_1$ and $k_2$ are sufficiently large depending only on $\gamma$. This implies that $k_2 \lesssim k_1$, and therefore $\left\lvert k_1-k_2\right\rvert \leq C$ provided that $C$ is sufficiently large (depending on $\gamma$). 
 \end{proof} 

\begin{lemma} \label{punity} There exist constants $\varrho>0$, $\varepsilon \in (0,1]$ and $J_1 \geq 0$ such that for all $J \geq J_1$, there is a partition of unity $\{\psi_{\tau}\}_{\tau \in \Lambda^J}$ subordinate to the cover $\left\{1.01 \tau : \tau \in \Lambda^J\right\}$ of
\[ \bigcup_{\tau \in \Lambda^{J,o}} \tau, \]
such that for each $\tau \in \Lambda^{J,o} \cap \Lambda_{j,k}$, the function $\psi_{\tau}$ is smooth and satisfies
\begin{multline} \label{steepness} \left\lvert \left(\frac{d}{dt}\right)^l \psi_{\tau}\left( x + t v \right) \right\rvert \lesssim_l \left\lvert 2^{-j} \langle v, \left(\gamma \times \gamma' \right)(\theta_{\tau} ) \rangle \right\rvert^l \\
+\left\lvert 2^{-j+k/2} \langle v, \gamma'(\theta_{\tau})  \rangle \right\rvert^l+ \left\lvert 2^{-j+k} \langle v,\gamma(\theta_{\tau}) \rangle \right\rvert^l,\end{multline}
for all $l \in \mathbb{N}$, $t \in \mathbb{R}$ and $x,v \in \mathbb{R}^3$. 

Given any $\delta>0$, for every $\tau \in \Lambda$ there exists a smooth partition of unity $\{\eta_T\}_{T \in \mathbb{T}_{\tau}}$ subordinate to the cover $\mathbb{T}_{\tau}=\mathbb{T}_{\tau}^{\delta} $ of $\mathbb{R}^3$, such that 
\begin{multline*} \left\lvert \left(\frac{d}{dt}\right)^l \eta_T\left( x + t v \right) \right\rvert \lesssim_l \left\lvert 2^{j-k -k\delta} \langle v, \gamma(\theta_{\tau} ) \rangle \right\rvert^l \\
+\left\lvert 2^{j-k/2-k\delta} \langle v, \gamma'(\theta_{\tau})  \rangle \right\rvert^l+ \left\lvert 2^{j-k/2-k\delta} \langle v, \left(\gamma\times \gamma' \right)(\theta_{\tau}) \rangle \right\rvert^l,\end{multline*}
for all $l \in \mathbb{N}$, $t \in \mathbb{R}$, $x,v \in \mathbb{R}^3$ and $T \in \mathbb{T}_{\tau}$.
\end{lemma}
\begin{proof} For $\tau \in \Lambda^+$, let
\[ g_{\tau}(x) = g_0\left( 4.005 + \left\langle x, \gamma(\theta_{\tau}) \right\rangle  \right), \]
and for $\tau \in \Lambda^-$, let
\[ g_{\tau}(x) = g_0\left( 4.005 -  \left\langle x, \gamma(\theta_{\tau}) \right\rangle  \right), \]
 where $g_0$ is a smooth function on $\mathbb{R}$ with $0 \leq g_0 \leq 1$, $g_0(x)=1$ for $x \geq 1/1000$ and $g(x) =0$ for $x \leq 0$. Choose $J_1$ large enough and $\varepsilon$ small enough to ensure that if $J \geq J_1$, if $\tau \in \Lambda^J$ and if $\tau' \in \Lambda \cap \Lambda_{j,j}$, then $g_{\tau'}(x) = 1$ for all $x \in \tau \cap (1.01)\tau'
$; such a choice of $J_1$ and $\varepsilon$ exists by the angle condition \eqref{betteranglebound} in Lemma~\ref{finiteoverlap}, and by \eqref{generalisedmvt}.

By translating and rescaling a fixed bump function on the unit cube, for each $\tau \in \Lambda$ let $f_{\tau}$ be a smooth bump function which is equal to 1 on $\tau$, nonzero in the interior of $1.01 \tau$, with $f_{\tau} \geq 1/100$ on $1.009\tau$ and with $f_{\tau}=0$ outside $1.01\tau$.  Let $J \geq J_1$. If $\tau \in \Lambda_{j,k} \cap \Lambda^{J}$ with $k < j$, let 
\begin{equation} \label{psidefn1} \psi_{\tau}(x) = \begin{cases} \frac{ f_{\tau}(x) }{ \sum_{\tau' \in \Lambda^J } f_{\tau'}(x)} &x \in (1.01\tau)^o \\
0 &x \in \mathbb{R}^3 \setminus (1.01\tau)^o. \end{cases}  \end{equation}
For $\tau \in \Lambda_{j,j} \cap \Lambda^J$, let 
\begin{equation} \label{psidefn2} \psi_{\tau}(x) = \begin{cases} \frac{ f_{\tau}(x)g_{\tau}(x) }{ \sum_{\tau' \in \Lambda^J } f_{\tau'}(x)} &x \in (1.01\tau)^o \\
0 &x \in \mathbb{R}^3 \setminus (1.01\tau)^o. \end{cases}  \end{equation}
Then $\sum_{\tau \in \Lambda^J} \psi_{\tau}(x) = 1$ for $x \in \bigcup_{\tau \in \Lambda^J} \tau$, by the choice of $J_1$ and $\varepsilon$.

It will be shown that if the constant $\varrho$ in \eqref{constantchoice} is small enough, and if $J_1$ is large enough, then for any $\tau \in \Lambda_{j,k} \cap \Lambda^{J,o}$ with $k < j$, 
\begin{equation} \label{overlap} 1.01 \tau  \subseteq \bigcup_{\tau' \in \Lambda^J} \tau'. \end{equation}
If $\tau \in \Lambda^+$ (which can be assumed; the argument for $\tau \in \Lambda^-$ being similar), this follows from the following argument. Given
\begin{align*} 1.01 \tau &\ni x \\
&=\lambda_1 \left(\gamma \times \gamma' \right)(\theta_{\tau} ) + \lambda_2 \gamma'(\theta_{\tau} ) +  \lambda_3 \gamma(\theta_{\tau} )  \\
&= \lambda_1 \left(\gamma \times \gamma' \right)\left( \theta_{\tau}  + \frac{\lambda_2}{\lambda_1 \det\left( \gamma(\theta_{\tau}),\gamma'(\theta_{\tau}), \gamma''(\theta_{\tau}) \right) }  \right)+ O\left(2^{j-k}\right), \end{align*}
let $\theta' = \theta_{\tau}  + \frac{\lambda_2}{\lambda_1 \det\left( \gamma(\theta_{\tau}),\gamma'(\theta_{\tau}), \gamma''(\theta_{\tau}) \right)} $. Choose $j'$ such that $2^{j'-1} \leq \lambda_1 \leq 2^{j'+1}$, and define $k'$ by 
\[ 2^{j'-k'} < \dist\left( x, \spn\left\{ \left( \gamma \times \gamma' \right)(\theta'), \gamma'(\theta') \right\} \right) \leq 2^{j'-k'+1}. \]
If $J_1$ is chosen sufficiently large and $\varepsilon$ is sufficiently small, then (by \eqref{generalisedmvt}) the parameter $k'$ is well-defined and satisfies $\left\lvert k-k'\right\rvert \lesssim 1$, $k' \geq 0$, and moreover  
\[ \left\langle x, \gamma(\theta') \right\rangle <0. \]
 If $k' \leq j'$, choose $\theta_{\tau'} \in \Theta_{k'}$ such that 
\[ \left\lvert \theta'  - \theta_{\tau'} \right\rvert \leq  2^{-k'/2} \varrho. \] 
Then (by \eqref{generalisedmvt}) if $J_1$ is sufficiently large and then $\varrho$ is chosen sufficiently small (depending on $\varepsilon$):
\begin{align*} 2^{j'-2} &\leq  \left\langle x, \left( \gamma \times \gamma' \right)(\theta_{\tau'}) \right\rangle \leq 2^{j'+2}; \\
 \left\lvert \left\langle x, \gamma'(\theta_{\tau'}) \right\rangle\right\rvert &\leq \varepsilon2^{j'-k'/2};  \\
 -2^{j'-k'+2} \leq \left\langle x, \gamma(\theta_{\tau'}) \right\rangle &\leq -2^{j'-k'-2}.  \end{align*}
By letting $\tau' \in \Lambda_{j',k'} \cap \Lambda^+$ be the box corresponding to the angle $\theta_{\tau'}$, this proves \eqref{overlap}. If $k' \geq j'$, the above argument still works by taking $\theta_{\tau'} \in \Theta_{j'}$ instead. The covering property in \eqref{overlap} implies that the denominator in the definition of $\psi_{\tau}$ in \eqref{psidefn1} is bounded away from zero on the support of the numerator, and therefore (by Lemma~\ref{finiteoverlap}) $\psi_{\tau}$ is smooth and satisfies the inequalities in \eqref{steepness} whenever $\tau \in  \Lambda^{J,o} \cap \Lambda_{j,k}$ with $k < j$. For the first part of the lemma, it remains to prove \eqref{steepness} in the case $k=j$.

For the case $k=j$, it will be shown that if the constant $\varrho$ in \eqref{constantchoice} is small enough, and if $J_1$ is large enough, then for any $\tau \in \Lambda_{j,j} \cap \Lambda^{J,o}$, 
\begin{equation} \label{overlap2} (1.01 \tau) \setminus \{ x:  \langle x, \gamma(\theta_{\tau} ) \rangle > 4.005 \}   \subseteq \bigcup_{\tau' \in \Lambda^J} 1.009\tau'. \end{equation}
To see this, given $x \in \left(1.01 \tau\right) \setminus \{ x:  \langle x, \gamma(\theta_{\tau} ) \rangle > 4.005 \}$, write
\begin{align*}  x &=\lambda_1 \left(\gamma \times \gamma' \right)(\theta_{\tau} ) + \lambda_2 \gamma'(\theta_{\tau} ) +  \lambda_3 \gamma(\theta_{\tau} )  \\
&= \lambda_1 \left(\gamma \times \gamma' \right)\left( \theta_{\tau}  + \frac{\lambda_2}{\lambda_1 \det\left( \gamma(\theta_{\tau}),\gamma'(\theta_{\tau}), \gamma''(\theta_{\tau}) \right) }  \right)+ O\left(1\right). \end{align*}
Choose $j'$ such that $2^{j'-1} \leq \lambda_1 \leq 2^{j'+1}$.
If 
\begin{equation} \label{tis} -4 \leq \langle x, \gamma(\theta_{\tau} ) \rangle \leq 4.005, \end{equation}
then let $k'=j'$ and choose $\theta_{\tau'} \in \Theta_{j'}$ such that 
\[ \left\lvert \theta'  - \theta_{\tau'} \right\rvert \leq 2^{-j'/2} \varrho . \] 
Then (by \eqref{generalisedmvt}) if $J_1$ is sufficiently large and then $\varrho$ is chosen sufficiently small (depending on $\varepsilon$):
\begin{align*} 2^{j'-2} \leq  \left\langle x, \left(\gamma \times \gamma'\right)(\theta_{\tau'}) \right\rangle &\leq 2^{j'+2}, \\
\left\lvert \left\langle x, \gamma'(\theta_{\tau'}) \right\rangle\right\rvert &\leq \varepsilon2^{j'/2}, \\
-4.036 \leq \left\langle x, \gamma(\theta_{\tau'}) \right\rangle  &\leq 4.036.  \end{align*}
If \eqref{tis} does not hold, then
\[ -4.04 \leq \langle x, \gamma(\theta_{\tau} ) \rangle \leq -4, \]
In this case, let $k' = j'-1$, and choose $\theta_{\tau'} \in \Theta_{k'}$ such that 
\[ \left\lvert \theta'  - \theta_{\tau'} \right\rvert \leq 2^{-k'/2}  \varrho . \] 
Then (by \eqref{generalisedmvt}) if $J_1$ is sufficiently large and then $\varrho$ is chosen sufficiently small (depending on $\varepsilon$):
\begin{align*} 2^{j'-2} \leq  \left\langle x, \left(\gamma \times \gamma'\right)(\theta_{\tau'}) \right\rangle &\leq 2^{j'+2}, \\
 \left\lvert \left\langle x, \gamma'(\theta_{\tau'}) \right\rangle\right\rvert &\leq \varepsilon2^{j'-k'/2}, \\
-2^{j'-k'+2} \leq \left\langle x, \gamma(\theta_{\tau'}) \right\rangle &\leq  -2^{j'-k'-2}.  \end{align*}
In either case, by letting $\tau' \in \Lambda_{j',k'} \cap \Lambda^+$ be the cap corresponding to angle $\theta_{\tau'}$, this proves \eqref{overlap2}. This implies that the denominator in the definition of $\psi_{\tau}$ in \eqref{psidefn2} is bounded away from zero on the support of the numerator, and therefore (by Lemma~\ref{finiteoverlap}) $\psi_{\tau}$ is smooth and satisfies \eqref{steepness} whenever $\tau \in \Lambda_{j,j} \cap \Lambda^{J,o}$. This proves the first part of the lemma.

The second part of the lemma is straightforward. \end{proof}

\begin{definition}  Let $\delta>0$. Let $\varepsilon, \varrho>0$ and $J_1 \geq 0$ be parameters ensuring the existence of the partition of unity in Lemma~\ref{punity}, and let $J \geq J_1$. Given a box $\tau \in \Lambda_{j,k} \cap \Lambda^J$  and $T \in \mathbb{T}_{\tau} = \mathbb{T}_{\tau}^{\delta}$, define 
\[ M_T f = \eta_T \left( f \ast \widecheck{\psi_{\tau} } \right), \]
for each Schwartz function $f$.  \end{definition}  
Let $\phi$ be a bump function equal to 1 on $B_3(0,1)$ which vanishes outside $B_3(0,2)$.
\begin{lemma} \label{decomp}  Let $\varepsilon, \varrho>0$ and $J_1\geq 0$ be parameters ensuring the existence of the partition of unity in Lemma~\ref{punity}, and let $J \geq J_1$. Let $j_0 \in \mathbb{N}$, let 
\[ \phi_{j_0}(x) = 2^{3j_0} \phi(2^{j_0}x ), \quad x \in \mathbb{R}^3, \]
let $\alpha \in [0,3]$, and let $\epsilon, \delta, \alpha_0'>0$. For any finite Borel measure $\mu$ on $B_3(0,1)$, and any $J \geq J_1$, there is a decomposition 
\[ \mu \ast \phi_{j_0} = \mu_g + \mu_b, \]
where 
\[ \mu_g = \mu_{g, j_0, J, \alpha, \epsilon, \delta, \alpha_0'} \quad \text{ and } \quad \mu_b = \mu_{b, j_0, J, \alpha, \epsilon, \delta, \alpha_0'}, \]
are complex-valued continuous functions supported in $B_3(0, 2^{J\delta})$, and
\begin{equation} \label{Linfty} \mu_b = \sum_{j> 2J} \sum_{\substack{k \in [j \epsilon, j] \\ k > J }} \sum_{\tau \in \Lambda_{j,k}} \sum_{T \in \mathbb{T}_{\tau,b}} M_T\left( \mu \ast \phi_{j_0} \right), \end{equation}
where 
\[ \mathbb{T}_{\tau,b} := \left\{ T \in \mathbb{T}_{\tau} : \mu(4T) \geq 2^{10 - k\alpha_0'/2-\alpha(j-k)} \right\}, \qquad \mathbb{T}_{\tau,g} = \mathbb{T}_{\tau} \setminus \mathbb{T}_{\tau,b}, \]
and the sum in \eqref{Linfty} converges in $L^{\infty}(\mathbb{R}^3)$. 
\end{lemma}
In the proof of the main theorem, only the behaviour of $\mu$ on tubes of radius at least $2^{-j_0}$ is considered, so there is no loss in convolving $\mu$ with the bump function above, and this (crucially) localises the frequencies to the ball of radius $\approx 2^{j_0}$. 

The precise exponent in the ``bad'' part $\mu_b$ is defined as above in such a way that the average $L^1$ norm of the measures $\pi_{\theta \#} \mu_b$ can be controlled by re-using Definition \ref{alpha0defn}, using the strategy below.  
\begin{proof}[Proof of Lemma~\ref{decomp}] Most of the lemma follows by defining $\mu_b$ as in \eqref{Linfty} and by defining $\mu_g = \mu \ast \phi_{j_0}-\mu_b$; the only nontrivial thing to check is that the sum in \eqref{Linfty} converges in $L^{\infty}$. For this it suffices to show that for any $T$ with $\tau(T) \in \Lambda_{j,k}$ and $j \geq j_0$,
\[ \left\lVert M_T\left( \mu \ast \phi_{j_0} \right) \right\rVert_{L^{\infty}} \lesssim_N 2^{(j_0-j)N}, \]
for any $N \geq 0$. By Hausdorff-Young, the definition of $M_T$, and the assumption that $\mu$ is finite, it suffices to show that for $j \geq 2j_0$,
\begin{equation} \label{schwartzdecay} \left\lVert \psi_{\tau} \cdot \widehat{\phi_{j_0} } \right\rVert_1 \lesssim_N 2^{-jN}. \end{equation}
By the Schwartz property of $\phi$, 
\[ \left\lVert \psi_{\tau} \cdot \widehat{\phi_{j_0} } \right\rVert_1 \leq \left\lVert \widehat{\phi_{j_0} } \right\rVert_{L^1(1.01\tau) } \lesssim_N \frac{m(\tau)}{2^{(j-j_0)N} }, \]
where $m(\tau)$ denotes the Lebesgue measure of $\tau$. By replacing $N$ with $3N$, this gives \eqref{schwartzdecay} and proves the lemma. 
\end{proof}
\begin{lemma} \label{pushtwo} Let $\delta>0$. Let $\varepsilon, \varrho>0$ and $J_1 \geq 0$ be parameters ensuring the existence of the partition of unity in Lemma~\ref{punity}. If $J \geq J_1$ is sufficiently large, and if $\tau \in \Lambda_{j,k} \cap \Lambda^{J,o}$, then 
\[ \left\lVert M_T (\mu \ast \phi_{j_0} ) \right\rVert_1 \lesssim_N 2^{3j \delta} \mu(2T) + \min\left\{2^{-JN}, 2^{-(j-j_0)N} \right\} \mu\left( \mathbb{R}^3\right), \]
 for any $T \in \mathbb{T}_{\tau}$ and any $N \geq 1$. \end{lemma}
\begin{proof} If $j > j_0 + J/10$, the inequality follows by Cauchy-Schwarz, Plancherel and the Schwartz decay of $\phi$. Assume then that $j \leq j_0 + J/10$. By the definition of $M_T(\mu \ast \phi_{j_0} )$, 
\begin{multline*} \left\lVert M_T \mu \right\rVert_1 \leq \int_{1.5T} \int_{y+2^{j\delta} \widehat{\tau} } \left\lvert \widecheck{\psi_{\tau} } (x-y) \right\rvert \, dx \, d(\mu \ast \phi_{j_0} )(y) \\
+ \int_{1.5T} \int_{T \setminus \left(y+2^{j\delta} \widehat{\tau} \right)} \left\lvert \widecheck{\psi_{\tau} } (x-y) \right\rvert \, dx \, d(\mu \ast \phi_{j_0} )(y) \\
+ \int_{\mathbb{R}^3 \setminus (1.5T) } \int_T \left\lvert \widecheck{\psi_{\tau} } (x-y) \right\rvert \, dx \, d(\mu \ast \phi_{j_0} )(y), \end{multline*} 
where $\widehat{\tau}$ is the ``dual'' box to $\tau$ centred at the origin; with axes parallel to $\tau$ but reciprocal side lengths. The first integral is $\lesssim 2^{3j \delta} (\mu \ast \phi_{j_0} )(1.5T)$, which is smaller than $2^{3j \delta} \mu(2T)$ since $k > J$ and $j \leq j_0 + J/10$. The second integral is $\lesssim_N 2^{-jN} (\mu \ast \phi_{j_0} )(1.5T)$ by Lemma~\ref{punity} and repeated integration by parts. Similarly, by Lemma~\ref{punity} and repeated integration by parts, the third integral is
\[ \lesssim_{N'}  \int_{\mathbb{R}^3 \setminus (1.5T) } \dist(y, T)^{-N'} \, d\mu(y) \lesssim 2^{-jN} \mu\left( \mathbb{R}^3\right), \]
if $N'$ is chosen large enough.  This proves the lemma. \end{proof}

\begin{lemma} \label{nonstat} Let $\delta>0$. Let $\varepsilon, \varrho>0$ and $J_1 \geq 0$ be parameters ensuring the existence of the partition of unity in Lemma~\ref{punity}, and let $J \geq J_1$. Then there exists $K_1 \geq 0$ such that if $\tau \in \Lambda_{j,k} \cap \Lambda^J$ and 
\begin{equation} \label{angleassumption2} \left\lvert  \theta - \theta_{\tau} \right\rvert \geq 2^{-k(1/2 - \delta)}, \quad \theta \in [a,b], \end{equation}
then for $k \geq K_1$, $N \geq 1$, $T \in \mathbb{T}_{\tau}$ and for $f \in L^1(\mathbb{R}^3)$, 
\[ \left\lVert \pi_{\theta \#} M_T f \right\rVert_{L^1(\mathbb{R}^3, \mathcal{H}^2 )} \lesssim_N 2^{-kN} m(\tau) \left\lVert f\right\rVert_1. \] \end{lemma} 
\begin{proof} By identifying the complex measure $\pi_{\theta \#} M_Tf$ with its Radon-Nikodym derivative with respect to $\mathcal{H}^2$, 
\begin{multline*}  \pi_{\theta \#} M_Tf (x) = \\
\int_{\mathbb{R}^3} f(y) \left[ \int_{\mathbb{R}^3}  \psi_{\tau}(\xi) e^{ 2\pi i \langle \xi, x-y \rangle } \left[ \int_{\mathbb{R} } \eta_T(x + t\gamma(\theta) ) e^{2 \pi i t \langle \xi, \gamma(\theta) \rangle } \, dt \right] \, d\xi \right] \, dy, \end{multline*} 
for any $x \in \gamma(\theta)^{\perp}$. It therefore suffices to show that
\[ \left\lvert \int_{\mathbb{R} } \eta_T(x + t\gamma(\theta) ) e^{2 \pi i t \langle \xi, \gamma(\theta) \rangle } \, dt \right\rvert \lesssim_N 2^{-kN}, \quad \forall \,  \xi \in \tau, \quad \forall \, x \in \mathbb{R}^3.  \] 
By repeated integration by parts, it suffices to show that for all $t \in \mathbb{R}$, $l \geq 1$, $\xi \in \tau $ and $x \in \mathbb{R}^3$,
\begin{equation} \label{temporary} \left\lvert \left( \frac{d}{dt} \right)^l \eta_T(x + t\gamma(\theta) ) \right\rvert \lesssim_l 2^{-lk\delta} \left\lvert \langle \xi, \gamma(\theta) \rangle \right\rvert^l. \end{equation}
Define $\varepsilon = \left\lvert \theta- \theta_{\tau} \right\rvert$. The assumed lower bound \eqref{angleassumption2} on $\varepsilon$, together with \eqref{generalisedmvt} and the assumption that $\gamma$ is $C^2$ with $\det(\gamma, \gamma', \gamma'')$ nonvanishing, yields 
\[ \left\lvert \langle \xi, \gamma(\theta) \rangle \right\rvert \gtrsim 2^j\varepsilon^2, \quad \forall \, \xi \in \tau, \]
 provided $k$ is sufficiently large. Hence to prove \eqref{temporary} it suffices to show that
\[ \left\lvert \left( \frac{d}{dt} \right)^l \eta_T(x + t \gamma(\theta) ) \right\rvert \lesssim_l \left(2^{j-k\delta} \varepsilon^2\right)^l \quad \forall \, l \geq 1. \]
By Lemma~\ref{punity} and \eqref{generalisedmvt},
\begin{align*}  \left\lvert \left( \frac{d}{dt} \right)^l \eta_T(x + t \gamma(\theta) ) \right\rvert   &\lesssim_l \left\lvert \left\langle \gamma(\theta), \gamma(\theta_{\tau}) \right\rangle 2^{j-k-k\delta} \right\rvert^l + \left\lvert \left\langle \gamma(\theta),\gamma'(\theta_{\tau}) \right\rangle 2^{j-k/2-k\delta} \right\rvert^l \\
&\quad +\left\lvert \left\langle \gamma(\theta), \left(\gamma \times \gamma' \right)(\theta_{\tau}) \right\rangle 2^{j-k/2 -k\delta} \right\rvert^l  \\
&\lesssim \left( 2^{j-k-k\delta} \right)^l  + \left( \varepsilon 2^{j-k/2-k\delta} \right)^l+ \left( \varepsilon^2 2^{j-k/2-k\delta} \right)^l \\
 &\lesssim \left( \varepsilon^2 2^{j-k\delta} \right)^l,  \end{align*}  
where the last line follows from the assumed lower bound \eqref{angleassumption2} on $\varepsilon$. \end{proof}

\section{Proof of the main theorem} \label{mainthmproof}
For a Borel measure $\mu$ on $\mathbb{R}^3$ and $\alpha \in [0,3]$, let $c_{\alpha}(\mu) = \sup_{\substack{ x \in \mathbb{R}^3 \\ r >0 } } \frac{ \mu(B(x,r))}{r^{\alpha}}$. 
\begin{definition} \label{alpha0defn} Let $\gamma: [a,b] \to S^2$ be a function and let $\alpha \in [0,3]$. Define $\alpha_0= \alpha_0(\alpha,\gamma)$ to be the supremum over all $\alpha^* \geq 0$ such that there exists $\delta = \delta(\alpha, \alpha^*,\gamma)>0$ and $C=C(\alpha, \alpha^*,\gamma)>0$ such that
\begin{equation} \label{deltadefn} \int_a^b \left( \pi_{\theta \#} \mu \right)\left( \bigcup_{D \in \mathbb{D}_{\theta}} D \right) \, d\theta \leq  C  \mu(\mathbb{R}^3) R^{-\delta}, \end{equation}
for all Borel measures $\mu$ on the unit ball with $c_{\alpha}(\mu) \leq 1$, for any $R \geq 1$, and for any collection of sets $\{\mathbb{D}_{\theta} : \theta \in [a,b] \}$ such that the integrand of \eqref{deltadefn} is measurable, where, for each $\theta \in [a,b]$, $\mathbb{D}_{\theta}$ is a disjoint set of at most $\mu(\mathbb{R}^3)R^{\alpha^*/2}$ discs in $\pi_{\theta}(\mathbb{R}^3)$ of radius $R^{-1/2}$.  \end{definition} 

The proof of the main theorem will be broken up into several separate lemmas. Lemma~\ref{badpart} deals with the contribution from the ``bad'' part of the measure, whilst Lemmas \ref{taudistance}, \ref{easypart}, \ref{kneqjcase}, \ref{trivial} and \ref{goodpart} deal with the ``good'' part. Lemma~\ref{conversion} converts everything into a lower bound for $\alpha_0$ in Definition~\ref{alpha0defn}, which is then used to obtain the main theorem.
\begin{lemma} \label{badpart} Suppose that $\gamma: [a,b] \to S^2$ is a $C^2$ unit speed curve satisfying \eqref{angleassumption}, \eqref{meanvalue} and \eqref{pgc} on $[a,b]$, and let $\left[ \widetilde{a}, \widetilde{b} \right] \subseteq (a,b)$. Let $\alpha \in [0,3]$ and $\epsilon \in (0, 1/100)$. If $\alpha_0'\in [0, \alpha_0(\alpha,1, \gamma\restriction_{\left[ \widetilde{a},\widetilde{b}\right]}))$, then there exists $\delta'>0$ such that for any $\delta \in (0, \delta']$, there is a positive integer $J_0$ such that for all $j_0 \geq J \geq J_0$ with $J \in [(j_0\epsilon)/1000, 1000j_0\epsilon]$ and for all Borel measures $\mu$ on the unit ball with $c_{\alpha}(\mu) \leq 1$, 
\[\int_{\widetilde{a}}^{\widetilde{b}} \int \left\lvert \pi_{\theta \#} \mu_{b} \right\rvert \, d\mathcal{H}^2 \, d\theta  \leq 2^{-J\delta'} \mu\left( \mathbb{R}^3 \right), \]
where $\mu_b=\mu_{b, j_0, J, \alpha, \epsilon, \delta, \alpha_0'}$ is defined by \eqref{Linfty} with respect to $\gamma: [a,b] \to S^2$. 
\end{lemma}
\begin{proof} Let $\delta'' = \delta''(\alpha, \alpha_0') \in (0,1/100)$ be an exponent that works in \eqref{deltadefn} with $\alpha^*$ replaced by $\alpha_0' + 100\delta'$ and with $A=1$, for some positive $\delta'$, and after taking $\delta'$ smaller if necessary assume that $\delta' \leq (\delta'' \epsilon^2)/100$. Let $\delta \in (0, \delta']$ be given.  Let $j_0$ and $J$ be such that $j_0 \geq J \geq J_0$, where $J_0$ is implicity chosen sufficiently large (depending on $\delta$) so that the argument below holds.  By Lemma~\ref{decomp},
\begin{align} \label{badbound} &\int_{\widetilde{a}}^{\widetilde{b}} \int \left\lvert \pi_{\theta \#} \mu_b \right\rvert d\mathcal{H}^2 \, d\theta \\
\notag &\quad \leq \int_{\widetilde{a}}^{\widetilde{b}} \sum_{j > 2J} \sum_{\substack{ k \in [j\epsilon,j] \\ k > J } } \sum_{\tau \in \Lambda_{j,k}} \sum_{T \in \mathbb{T}_{\tau,b} }  \int \left\lvert \pi_{\theta \#} M_T\left(\mu \ast \phi_{j_0} \right) \right\rvert d\mathcal{H}^2 \, d\theta\\
&\label{firstterm} = \int_{\widetilde{a}}^{\widetilde{b}} \sum_{j > 2J} \sum_{\substack{ k \in [j\epsilon,j] \\ k > J } } \sum_{\substack{\tau \in \Lambda_{j,k}: \\
\left\lvert \theta_{\tau} - \theta \right\rvert < 2^{k(-1/2+\delta) }}} \sum_{T \in \mathbb{T}_{\tau,b} }   \int \left\lvert \pi_{\theta \#} M_T\left(\mu \ast \phi_{j_0} \right) \right\rvert d\mathcal{H}^2 \, d\theta \\
\label{secondterm} &\qquad + \int_{\widetilde{a}}^{\widetilde{b}} \sum_{j > 2J} \sum_{\substack{ k \in [j\epsilon,j] \\ k > J } } \sum_{\substack{\tau \in \Lambda_{j,k}: \\
\left\lvert \theta_{\tau} - \theta \right\rvert \geq 2^{k(-1/2+\delta) }}} \sum_{T \in \mathbb{T}_{\tau,b} }   \int \left\lvert \pi_{\theta \#} M_T\left(\mu \ast \phi_{j_0} \right) \right\rvert d\mathcal{H}^2 \, d\theta. \end{align}
By Lemma~\ref{nonstat}, the contribution from \eqref{secondterm} is
\[ \lesssim_{\delta, \epsilon} 2^{-J} \mu\left( \mathbb{R}^3 \right). \]
If $J_0$ is sufficiently large, then by Lemma~\ref{pushtwo} the contribution from \eqref{firstterm} is
\begin{align} \notag &\lesssim 2^{-J} \mu\left( \mathbb{R}^3 \right) + \sum_{j > 2J} \sum_{\substack{ k \in [j\epsilon,j] \\ k > J } } \int_{\widetilde{a}}^{\widetilde{b}}  \sum_{\substack{\tau \in \Lambda_{j,k}: \\
\left\lvert \theta_{\tau} - \theta \right\rvert < 2^{k(-1/2+\delta) }}} \sum_{T \in \mathbb{T}_{\tau,b} }  2^{3j\delta} \mu(2T) \, d\theta  \\
\label{intermediate} &\lesssim 2^{-J} \mu\left( \mathbb{R}^3 \right)+ \sum_{j > 2J} \sum_{\substack{ k \in [j\epsilon,j] \\ k > J } }   2^{10j \delta}\left( \mathcal{H}^1 \times \mu \right)(B_{j,k} ) , \end{align}
where $\mathcal{H}^1$ is the Lebesgue measure on $[a,b]$, 
\[ B_{j,k} = \left\{( \theta, x) \in \left[\widetilde{a}, \widetilde{b} \right] \times \mathbb{R}^3 : x \in B_{j,k}(\theta) \right\}, \]
and 
\[ B_{j,k}(\theta) = \bigcup_{\substack{ \tau \in \Lambda_{j,k}: \\ \left\lvert \theta_{\tau} - \theta \right\rvert < 2^{k(-1/2+\delta) }}} \bigcup_{T \in \mathbb{T}_{\tau,b} } 2T. \] 
For fixed $j$ and $k$, let $\{B_l \}_l$ be a finitely overlapping cover of $B_3(0,1)$ by balls of radius $2^{-(j-k)}$. For each $\theta$ and $l$ let 
\[ B_{j,k,l}(\theta) = \bigcup_{\substack{ \tau \in \Lambda_{j,k}: \\ \left\lvert \theta_{\tau} - \theta \right\rvert <2^{k(-1/2+\delta) }}} \bigcup_{\substack{T \in \mathbb{T}_{\tau,b}:  \\ 2T \cap B_l \neq \emptyset }} 2T, \]
and let
\[ B_{j,k,l} = \left\{( \theta, x) \in \left[\widetilde{a}, \widetilde{b} \right] \times \mathbb{R}^3 : x \in B_{j,k,l}(\theta) \right\}. \] Let $\mu_{j,k}$ be the pushforward of $\mu$ under $x \mapsto 2^{j-k-2k\delta} x$. Then 
\begin{align} \notag \left( \mathcal{H}^1 \times \mu \right)(B_{j,k} ) &\leq \sum_{l}  \left( \mathcal{H}^1 \times \mu \right)(B_{j,k,l} ) \\
\label{tildemu} &= \sum_{l} 2^{-\alpha(j-k-2k\delta)} \left( \mathcal{H}^1 \times \widetilde{\mu}_{j,k,l} \right)(B_{j,k,l}' ), \end{align} 
where 
\[ B_{j,k,l}' = \left\{( \theta, x) \in \left[\widetilde{a}, \widetilde{b} \right] \times \mathbb{R}^3 : x \in 2^{j-k-2k\delta} B_{j,k,l}(\theta) \right\}, \]
and $\widetilde{\mu}_{j,k,l} =2^{\alpha(j-k-2k\delta)} \cdot \mu_{j,k}\chi_{\widetilde{B}_l}$, where 
\[ \widetilde{B}_l = \left\{ 2^{j-k-2k\delta}b_l + y : \left\lvert y\right\rvert \leq 1 \right\}, \]
with $b_l$ the centre of $B_l$. Up to translation and finite overlaps, $B_{j,k,l}'$ and $\widetilde{\mu}_{j,k,l} $ satisfy the conditions of Definition~\ref{alpha0defn}; for each $\theta$ the set $2^{j-k-2k\delta}B_{j,k,l}(\theta)$ is contained in a union of tubes of radius $2^{-k/2}$ parallel to $\gamma(\theta)$, with the number of tubes $\lesssim 2^{k(\alpha_0'+100\delta')/2} \widetilde{\mu}_{j,k,l} (\mathbb{R}^3)$, such that each tube overlaps $\lesssim 2^{10k\delta}$ of the others. Moreover $c_{\alpha}\left(\widetilde{\mu}_{j,k,l}  \right) \leq 1$ and $\widetilde{\mu}_{j,k,l} $ is supported in a ball of radius 1. Hence
\[ \left( \mathcal{H}^1 \times \widetilde{\mu}_{j,k,l} \right)(B_{j,k,l}' ) \lesssim  \widetilde{\mu}_{j,k,l} \left(\mathbb{R}^3\right) 2^{-k\delta''/4} \\
\leq   2^{\alpha(j-k-2k\delta)}\mu_{j,k}\left(\widetilde{B}_l\right) 2^{-k\delta''/4}, \]
 Putting this into \eqref{tildemu} yields
\[ \left( \mathcal{H}^1 \times \mu \right)(B_{j,k} ) \lesssim  2^{-k\delta''/4}\mu\left( \mathbb{R}^3 \right), \]
Substituting this into \eqref{intermediate} and then \eqref{badbound} gives
\[ \int_{\widetilde{a}}^{\widetilde{b}} \int \left\lvert \pi_{\theta \#} \mu_b \right\rvert d\mathcal{H}^2 \, d\theta \lesssim 2^{-(J \epsilon\delta'')/8}\mu\left( \mathbb{R}^3 \right), \]
and hence
\[ \int_{\widetilde{a}}^{\widetilde{b}} \int \left\lvert \pi_{\theta \#} \mu_b \right\rvert d\mathcal{H}^2 \, d\theta \leq 2^{-J\delta'}\mu\left( \mathbb{R}^3 \right), \]
provided $J_0$ is sufficiently large. This proves the lemma. \end{proof}

\begin{lemma} \label{taudistance} Let $\gamma : [a,b] \to S^2$ be a $C^2$ unit speed curve satisfying \eqref{angleassumption}, \eqref{meanvalue} and \eqref{pgc} on $[a,b]$. Then there exists a constant $\sigma>0$ depending on $\gamma$, and for each $\epsilon \in (0,1)$ an integer $J_0 \geq 0$ depending on $\gamma$ and $\epsilon$, such that for any $\tau \in \Lambda^J \cap \Lambda_{j,k}$ with $J \geq J_0$ and $k \in [j \epsilon,j]$, 
\begin{equation} \label{distancebound} \dist\left( 1.01\tau,  \eta_1 \gamma'(\theta)+ \eta_2 \left( \gamma \times \gamma'\right)(\theta) \right) \geq  \sigma\max\left( \left\lvert \eta \right\rvert^{1-10\epsilon}, 2^{j(1-10\epsilon) } \right), \end{equation}
for all $\eta \in \mathbb{R}^2$ with $\left\lvert \eta_1\right\rvert \geq \left\lvert \eta_2\right\rvert^{1-\epsilon}$ and for all $\theta \in [a,b]$ with $\left\lvert \theta-\theta_{\tau}\right\rvert \leq \sigma$. \end{lemma}
\begin{proof}If either $\left\lvert \eta\right\rvert < 2^{j-10}$ or $\left\lvert \eta\right\rvert > 2^{j+10}$ this is immediate, so it may be assumed that $2^{j-10} \leq \left\lvert \eta\right\rvert \leq 2^{j+10}$. Let $x= \lambda_1 (\gamma \times \gamma')(\theta_{\tau} ) + \lambda_2 \gamma'(\theta_{\tau} ) + \lambda_3 \gamma(\theta_{\tau}) \in 1.01\tau$, where $\left\lvert \lambda_1 \right\rvert \sim 2^j$, $\left\lvert \lambda_2\right\rvert \lesssim 2^{j-k/2}$ and $\left\lvert \lambda_3\right\rvert \sim 2^{j-k}$. Suppose first that $\left\lvert \theta_{\tau} - \theta\right\rvert \leq 2^{-3j\epsilon}$. If $k \leq  5j \epsilon$ then $\left\lvert \theta_{\tau} - \theta\right\rvert \ll 2^{-k/2}$ and hence $\left\lvert \left\langle x, \gamma(\theta) \right\rangle \right\rvert \gtrsim 2^{j-k} \geq 2^{j(1-6\epsilon)}$, which implies that
\begin{equation} \label{firstone}  \dist\left( x,  \eta_1 \gamma'(\theta)+ \eta_2 \left( \gamma \times \gamma'\right)(\theta) \right) \gtrsim  \max\left( \left\lvert \eta \right\rvert^{1-10\epsilon}, 2^{j(1-10\epsilon) } \right). \end{equation}
If $k > 5 j \epsilon$ then $\left\lvert \langle x, \gamma'(\theta) \rangle \right\rvert \leq 2^{j(1-2\epsilon)}$, which, due to the condition $\left\lvert \eta_1\right\rvert \geq \left\lvert \eta_2\right\rvert^{1-\epsilon}$, implies \eqref{firstone}. It remains to consider the possibility that $\left\lvert \theta-\theta_{\tau}\right\rvert > 2^{-3j\epsilon}$, in which case
\[ \langle \lambda_1 ( \gamma \times \gamma') (\theta_{\tau} ) , \gamma(\theta) \rangle = \frac{\lambda_1}{2}\left[(\theta_{\tau}-\theta)^2 \det\left( \gamma(\theta_{\tau}), \gamma'(\theta_{\tau}), \gamma''(\theta_{\tau} ) \right) + o\left( \left\lvert \theta_{\tau} - \theta \right\rvert^2 \right) \right], \]
which gives $\left\lvert \left\langle x, \gamma(\theta) \right\rangle \right\rvert \gtrsim 2^{j(1-6\epsilon)}$ for $\left\lvert \theta-\theta_{\tau}\right\rvert < c_{\gamma}$, and this implies \eqref{firstone}.   \end{proof}

\begin{lemma} \label{easypart} Let $\gamma : [a,b] \to S^2$ be a $C^2$ unit speed curve satisfying \eqref{angleassumption}, \eqref{meanvalue} and \eqref{pgc} on $[a,b]$,  let $\left[\widetilde{a},\widetilde{b}\right] \subseteq (a,b)$. and assume that $\left\lvert \widetilde{b}-\widetilde{a}\right\rvert \leq \sigma$, where $\sigma$ is a constant that works in Lemma~\ref{taudistance}. Let $\epsilon \in (0,1)$, $\delta>0$, let $\alpha \in [0,3]$ and let $\alpha_0' >0$. Then there exists a positive integer $J_0$ such that for all $j_0 \geq J \geq J_0$ and for all finite Borel measures $\mu$ on the unit ball,
\begin{multline*} \int_{\widetilde{a}}^{\widetilde{b}} \int_{\{ \left\lvert \eta_1\right\rvert \geq \left\lvert \eta_2\right\rvert^{1-\epsilon} \} \cap B\left(0, 2^{j_0(1+\delta) } \right) } \left\lvert \widehat{\mu_g}\left( \eta_1 \gamma'(\theta)+ \eta_2 \left( \gamma \times \gamma'\right)(\theta) \right)\right\rvert^2 \, d\eta \, d\theta\\
 \leq 8\mu\left( \mathbb{R}^3 \right)^2 +  8\int_{\widetilde{a}}^{\widetilde{b}} \int_{\{ \left\lvert \eta_1\right\rvert \geq \left\lvert \eta_2\right\rvert^{1-\epsilon} \} \cap B\left(0, 2^{j_0(1+\delta) } \right)} \left\lvert \widehat{\mu}\left( \eta_1 \gamma'(\theta)+ \eta_2 \left( \gamma \times \gamma'\right)(\theta) \right)\right\rvert^2 \, d\eta \, d\theta, \end{multline*} 
where $\mu_g =  \mu_{g, j_0, J, \alpha, \epsilon, \delta, \alpha_0'}$ is defined by Lemma~\ref{decomp}.  \end{lemma}
\begin{proof} This follows from the triangle inequality, Lemma~\ref{taudistance}, the definition of $\mu_b$, and the rapid decay of $\mathcal{F}\left(M_T(\mu \ast \phi_{j_0})\right)$ outside $1.01\tau(T)$.  \end{proof} 

\begin{lemma} \label{kneqjcase} Let $\alpha \in [0,3]$, $\alpha_0' \in (0,3]$ and let $\epsilon \in (0,1/2)$. Suppose that $\gamma: [a,b] \to S^2$ is a $C^3$ unit speed curve satisfying \eqref{angleassumption}, \eqref{meanvalue} and \eqref{pgc}, on $[a,b]$, and let $\left[ \widetilde{a}, \widetilde{b} \right] \subseteq (a,b)$. Let
\[ A_{j,k} = \left[B(0, 2^{j+1}) \setminus B(0, 2^j) \right] \cap \{ 2^{j-k/2} \leq \left\lvert \eta_1\right\rvert < 2^{j+1-k/2} \}, \quad k < j, \]
and 
\[ A_{j,j}  = \left[B(0, 2^{j+1}) \setminus B(0, 2^j) \right] \cap \{ \left\lvert \eta_1\right\rvert < 2^{(j+1)/2}   \}. \]
Then there exists $\delta_0 \in (0, \epsilon^{100})$ and for any $\delta \in (0, \delta_0]$ a $J_0 \geq 0$ such that if $J \geq J_0$, $j_0(1+\delta) \geq j \geq 3J$ and $j\epsilon \leq k \leq j$, then for any Borel measure $\mu$ on $B_3(0,1)$ with $c_{\alpha}(\mu) \leq 1$,
\begin{multline} \label{jkintegral} \int_{\widetilde{a}}^{\widetilde{b}} \int_{A_{j,k}} \left\lvert \widehat{\mu_g}\left( \eta_1 \gamma'(\theta)+ \eta_2 \left( \gamma \times \gamma'\right)(\theta) \right)\right\rvert^2 \, d\eta \, d\theta \\
\leq \mu\left( \mathbb{R}^3 \right) 2^{100j\epsilon+ j(2-\alpha) + k \left( -\frac{1}{2} + \frac{2\alpha}{3} - \frac{\alpha_0'}{3} \right)}, \end{multline} 
where $\mu_b=\mu_{b, j_0, J, \alpha, \epsilon, \delta, \alpha_0'}$ is defined by \eqref{Linfty} with respect to $\gamma: [a,b] \to S^2$. 
 \end{lemma}
\begin{proof} Suppose first that $k < j$. By the wave packet decomposition and Lemma~\ref{finiteoverlap}, there is a constant $C$ such that
\begin{multline} \label{goodtubes} \eqref{jkintegral} \leq C2^{-100j}\mu(\mathbb{R}^3)^2 + \int_{\widetilde{a}}^{\widetilde{b}} \int_{A_{j,k}}  \\
 \left\lvert \sum_{\left\lvert j'-j\right\rvert \leq C } \sum_{\substack{\left\lvert k'-k\right\rvert \leq C \\ k' \leq j'}} \sum_{\tau \in \Lambda_{j',k'}} \sum_{T \in \mathbb{T}_{\tau, g} } \mathcal{F}\left(M_T \left( \mu \ast \phi_{j_0} \right) \right)\left(\eta_1 \gamma'(\theta)+ \eta_2 \left( \gamma \times \gamma'\right)(\theta) \right)\right\rvert^2 \, d\eta \, d\theta. \end{multline}  
Let $\widetilde{\mu}$ be defined by setting $\widehat{\widetilde{\mu}}$ equal to the function inside the modulus signs above. Let $\{B_m\}_m$ be a finitely overlapping cover of $\mathbb{R}^3$ by balls of radius $2^{k-j}$, and let $\{ \vartheta_m\}_m$ be a corresponding subordinate smooth partition of unity. Then by changing variables (see~\eqref{stp1}--\eqref{unitspeed} below) and by Plancherel,
\begin{equation} \label{localised} \eqref{goodtubes} \lesssim C2^{-100j}\mu(\mathbb{R}^3)^2+ 2^{k/2-j}\sum_m \int_{\mathbb{R}^3} \left\lvert \widehat{ \widetilde{\mu} \vartheta_m} \right\rvert^2. \end{equation}
For each $m$ and for arbitrarily large $N$, the integral in \eqref{localised} satisfies 
\begin{multline} \label{wavepackets}   \int_{\mathbb{R}^3} \left\lvert \widehat{ \widetilde{\mu} \vartheta_m} \right\rvert^2
\lesssim \\
C_N2^{-kN}\mu\left( \mathbb{R}^3 \right)^2 +  \sum_{\left\lvert j'-j\right\rvert \leq C } \sum_{\substack{\left\lvert k'-k\right\rvert \leq C \\ k' \leq j'}} \sum_{\tau \in \Lambda_{j',k'}} \sum_{\substack{T \in \mathbb{T}_{\tau, g}  \\ T \cap B_m \neq \emptyset }} \int \left\lvert M_T \left(\mu \ast \phi_{j_0} \right) \right\rvert^2,  \end{multline} 
which follows from the ``essential orthogonality'' of wave packets. 

A similar inequality will be shown in the case $k=j$. By the wave packet decomposition and Lemma~\ref{finiteoverlap},
\begin{multline*} \eqref{jkintegral} \leq C2^{-100j}\mu(\mathbb{R}^3)^2 +  \int_{\widetilde{a}}^{\widetilde{b}} \int_{A_{j,j}}  \\
 \left\lvert \sum_{\left\lvert j'-j\right\rvert \leq C } \sum_{\substack{\left\lvert k'-j\right\rvert \leq C \\ k' \leq j'}} \sum_{\tau \in \Lambda_{j',k'}} \sum_{T \in \mathbb{T}_{\tau, g} } \mathcal{F} \left( M_T \left( \mu \ast \phi_{j_0} \right)\right)\left(\eta_1 \gamma'(\theta)+ \eta_2 \left( \gamma \times \gamma'\right)(\theta) \right)\right\rvert^2 \, d\eta \, d\theta. \end{multline*}  
Let $\widetilde{\mu}$ be defined by setting $\widehat{\widetilde{\mu}}$ equal to the function inside the modulus signs above. By the finite overlapping property of the sets $\tau$, 
\begin{multline} \label{wavepackets2} \int_{\widetilde{a}}^{\widetilde{b}} \int_{A_{j,j}} \left\lvert \widehat{ \widetilde{\mu}}\left(\eta_1 \gamma'(\theta)+ \eta_2 \left( \gamma \times \gamma'\right)(\theta) \right) \right\rvert^2 \, d\eta \, d\theta \leq C2^{-100j}\mu(\mathbb{R}^3)^2
\\ + \sum_{\left\lvert j'-j\right\rvert \leq C } \sum_{\substack{\left\lvert k'-j\right\rvert \leq C \\ k' \leq j'}} \sum_{\tau \in \Lambda_{j',k'}}  \\
\int_{\widetilde{a}}^{\widetilde{b}} \int \left\lvert \sum_{\substack{T \in \mathbb{T}_{\tau, g}  \\ T \cap B(0,10) \neq \emptyset }} \sum_{T' \subseteq T} \mathcal{F}\left( M_{T'} \left(\mu \ast \phi_{j_0} \right)\right)\left(\eta_1 \gamma'(\theta)+ \eta_2 \left( \gamma \times \gamma'\right)(\theta) \right) \right\rvert^2 d\eta \, d\theta,  \end{multline} 
where the sets $T'$ cover $T$ with planks of dimensions $\approx 1 \times 2^{-j/2} \times 2^{-j}$, with long direction parallel to $\gamma(\theta_{\tau(T)})$, medium direction parallel to $\gamma'(\theta_{\tau(T)})$ and short direction parallel to $\left( \gamma \times \gamma'\right)(\theta_{\tau(T)} )$, and $M_{T'}\mu = \eta_{T'} M_T\mu$, where $\{\eta_{T'}\}_{T'}$ is a smooth partition of unity subordinate to the cover $\{T'\}_{T'}$. By the 2-dimensional Plancherel theorem followed by the uncertainty principle (bounding the $L^2$ norm by the $L^{\infty}$ norm, followed by Hausdorff-Young and Cauchy-Schwarz),
\begin{multline*} \eqref{wavepackets2} \lesssim C2^{-100j}\mu(\mathbb{R}^3)^2 \\
+ 2^{10j\delta-j/2} \sum_{\left\lvert j'-j\right\rvert \leq C } \sum_{\substack{\left\lvert k'-j\right\rvert \leq C \\ k' \leq j'}} \sum_{\tau \in \Lambda_{j',k'}} \sum_{\substack{T \in \mathbb{T}_{\tau, g}  \\ T \cap B(0,10) \neq \emptyset }} \int \left\lvert M_T \left(\mu \ast \phi_{j_0} \right) \right\rvert^2. \end{multline*} 
This shows that \eqref{localised}--\eqref{wavepackets} holds also in the case $k=j$, although possibly with a $2^{10j\delta}$ loss, and with $\{B_m\}_m$ equal to the cover of $B(0,1)$ by the single ball $B(0,10)$ in that case. The remainder of the proof will therefore cover both cases simultaneously ($k \leq j$). 

Applying Plancherel to the non-negligible term in the right hand side of \eqref{wavepackets} gives 
\begin{multline} \label{plancherel} \sum_{\left\lvert j'-j\right\rvert \leq C } \sum_{\substack{\left\lvert k'-k\right\rvert \leq C \\ k ' \leq j'}} \sum_{\tau \in \Lambda_{j',k'}} \sum_{\substack{T \in \mathbb{T}_{\tau, g}  \\ T \cap B_m \neq \emptyset }} \int \left\lvert M_T \left(\mu \ast \phi_{j_0} \right) \right\rvert^2 \\
= \int \sum_{\left\lvert j'-j\right\rvert \leq C } \sum_{\substack{\left\lvert k'-k\right\rvert \leq C \\ k ' \leq j'}} \sum_{\tau \in \Lambda_{j',k'}} \sum_{\substack{T \in \mathbb{T}_{\tau, g}  \\ T \cap B_m \neq \emptyset }}\left(\left[ \eta_T M_T \left(\mu \ast \phi_{j_0} \right) \right] \ast \widecheck{\psi_{\tau} }\right) d\left( \mu \ast \phi_{j_0} \right). \end{multline} 
Let $\nu$ be the restriction of $\mu \ast \phi_{j_0}$ to $2^{10k\delta}B_m$. By Cauchy-Schwarz,
\begin{multline} \label{above} \eqref{plancherel} \leq C_N2^{-kN} \mu\left( \mathbb{R}^3 \right)^2 +  \mu\left( 2^{100k\delta}B_m \right)^{1/2} \times \\
 \left(\int \left\lvert \sum_{\left\lvert j'-j\right\rvert \leq C } \sum_{\substack{\left\lvert k'-k\right\rvert \leq C \\ k ' \leq j'}} \sum_{\tau \in \Lambda_{j',k'}} \sum_{\substack{T \in \mathbb{T}_{\tau, g}  \\ T \cap B_m \neq \emptyset }}\left[ \eta_T M_T \left(\mu \ast\phi_{j_0} \right) \right] \ast \widecheck{\psi_{\tau} } \right\rvert^2 d\nu\right)^{1/2}. \end{multline} 
Let 
\[ f_T = \left[ \eta_T M_T \left(\mu \ast\phi_{j_0} \right) \right] \ast \widecheck{\psi_{\tau} }, \qquad  f = \sum_{\left\lvert j'-j\right\rvert \leq C } \sum_{\substack{\left\lvert k'-k\right\rvert \leq C \\ k ' \leq j'}} \sum_{\tau \in \Lambda_{j',k'}} \sum_{\substack{T \in \mathbb{T}_{\tau, g}  \\ T \cap B_m \neq \emptyset }}f_T.   \]
The integral in \eqref{above} satisfies 
\[ \int \left\lvert f \right\rvert^2 d\nu \lesssim \int \left\lvert f \right\rvert^2 d(\nu \ast \zeta), \]
where $\zeta(x) = \frac{ 2^{3j}}{1+2^{jN}\left\lvert x\right\rvert^N}$ for some very large $N$. This follows from the uncertainty principle since $\widehat{f}$ is supported in a ball of radius $\lesssim 2^j$. By dyadic pigeonholing, there is a subset
\[ \mathbb{W} \subseteq  \bigcup_{\left\lvert j'-j\right\rvert \leq C } \bigcup_{\substack{\left\lvert k'-k\right\rvert \leq C \\ k ' \leq j'}} \bigcup_{\tau \in \Lambda_{j',k'}}\left\{ T \in \mathbb{T}_{\tau, g}: T \cap B_m \neq \emptyset  \right\}, \]
such that $\lVert f_T \rVert_2$ is constant up to a factor of 2 as $T$ varies over $\mathbb{W}$, and
\begin{multline*} \int \left\lvert f \right\rvert^2 d(\nu \ast \zeta) \lesssim \log\left( 2^j \right)^2  \int \left\lvert \sum_{T \in \mathbb{W}} f_T \right\rvert^2 d(\nu \ast \zeta) \\
+2^{-100j} \sum_{\left\lvert j'-j\right\rvert \leq C } \sum_{\substack{\left\lvert k'-k\right\rvert \leq C \\ k ' \leq j'}} \sum_{\tau \in \Lambda_{j',k'}} \sum_{\substack{T \in \mathbb{T}_{\tau, g}  \\ T \cap B_m \neq \emptyset }} \lVert f_T \rVert_2^2. \end{multline*}
By pigeonholing again and by Hölder's inequality, there is a disjoint union $Y$ of balls $Q$ of radius $2^{-j+k/2}$, such that 
\begin{equation} \label{holder} \int \left\lvert \sum_{T \in \mathbb{W}} f_T \right\rvert^2 d(\nu \ast \zeta) \lesssim \log\left( 2^k \right) \left\lVert f \right\rVert_{L^6(Y) }^2 \left( \int_Y \left( \nu \ast \zeta\right)^{3/2} \right)^{2/3},  \end{equation}
and such that each $Q \subseteq Y$ intersects a number $\# \in [M,2M)$ boxes $3T$ as $T$ varies over $\mathbb{W}$, for some dyadic number $M$. By rescaling and then applying the refined Strichartz inequality (Theorem~\ref{refinedstrichartz}) with $p=6$, the first factor in \eqref{holder} satisfies
\[ \left\lVert f \right\rVert_{L^6(Y) } \leq C_{\epsilon,\delta} 2^{j-k/2+k\epsilon} \left( \frac{M}{\left\lvert \mathbb{W} \right\rvert } \right)^{1/3} \left( \sum_{T \in \mathbb{W} } \left\lVert f_T \right\rVert_2^2 \right)^{1/2}. \]
For the second factor in \eqref{holder}, the assumed inequality $c_{\alpha}(\mu) \leq 1$ implies that $\left\lVert \nu \ast \zeta\right\rVert_{\infty} \lesssim 2^{j(3-\alpha)}$. Hence 
\begin{align*} \int_Y \left( \nu \ast \zeta\right)^{3/2} &\lesssim \frac{2^{\frac{j}{2}(3-\alpha)} }{M} \sum_{T \in \mathbb{W} } \left( \nu \ast \zeta \right)(3.5T) \\
&\leq \frac{2^{\frac{j}{2}(3-\alpha)} }{M} \sum_{T \in \mathbb{W} }  \mu (4T) + C_N 2^{-jN} \mu\left(\mathbb{R}^3 \right) \\
&\lesssim \frac{2^{\frac{j}{2}(3-\alpha)} \left\lvert \mathbb{W} \right\rvert 2^{-k\alpha_0'/2 -\alpha(j-k)} }{M} + C_N 2^{-jN}. \end{align*} 
Hence 
\begin{equation} \label{combined} \eqref{holder} \lesssim \log(2^k) 2^{j(3-\alpha) +k\left(-1 + \frac{2\alpha}{3} - \frac{\alpha_0'}{3}+\epsilon \right)} \sum_{T \in \mathbb{W}} \left\lVert f_T \right\rVert_2^2. \end{equation}
By Plancherel, $\left\lVert f_T\right\rVert_2 \leq \left\lVert M_T \left( \mu \ast \phi_{j_0} \right) \right\rVert_2$ for every $T$. Assuming the tail terms are not dominant, substituting into \eqref{combined} and then into \eqref{plancherel} yields
\begin{multline*} \sum_{\left\lvert j'-j\right\rvert \leq C } \sum_{\substack{\left\lvert k'-k\right\rvert \leq C \\ k ' \leq j'}} \sum_{\tau \in \Lambda_{j',k'}} \sum_{\substack{T \in \mathbb{T}_{\tau, g}  \\ T \cap B_m \neq \emptyset }} \int \left\lvert M_T \left( \mu \ast \phi_{j_0} \right) \right\rvert^2 \\
\lesssim \mu\left( 2^{100k\delta} B_m \right)^{1/2} 2^{10j\epsilon+ \frac{1}{2} \left[ j(3-\alpha) + k \left( -1 + \frac{2\alpha}{3} - \frac{\alpha_0'}{3} \right) \right] } \\
\times \left( \sum_{\left\lvert j'-j\right\rvert \leq C } \sum_{\substack{\left\lvert k'-k\right\rvert \leq C \\ k ' \leq j'}} \sum_{\tau \in \Lambda_{j',k'}} \sum_{\substack{T \in \mathbb{T}_{\tau, g}  \\ T \cap B_m \neq \emptyset }} \int \left\lvert M_T \left( \mu \ast \phi_{j_0} \right) \right\rvert^2  \right)^{1/2}. \end{multline*}  
By cancelling the common factors, this yields 
\begin{multline*} \sum_{\left\lvert j'-j\right\rvert \leq C } \sum_{\substack{\left\lvert k'-k\right\rvert \leq C \\ k ' \leq j'}} \sum_{\tau \in \Lambda_{j',k'}} \sum_{\substack{T \in \mathbb{T}_{\tau, g}  \\ T \cap B_m \neq \emptyset }} \int \left\lvert M_T \left( \mu \ast \phi_{j_0} \right) \right\rvert^2 \\
\lesssim \mu\left( 2^{100k\delta} B_m \right) 2^{20j\epsilon+ j(3-\alpha) + k \left( -1 + \frac{2\alpha}{3} - \frac{\alpha_0'}{3} \right)}. \end{multline*}  
By substituting back into \eqref{localised}--\eqref{wavepackets} and summing over $m$, this proves the lemma (the $m$ for which the tail terms dominate make a negligible contribution to the sum). \end{proof} 

\begin{lemma} \label{trivial} Let $\alpha \in [0,3]$, $\alpha_0' >0$, and let $\epsilon, \delta>0$ with $\delta < \epsilon/1000$. Let $\gamma: [a,b] \to S^2$ be a $C^2$ unit speed curve satisfying \eqref{angleassumption}, \eqref{meanvalue} and \eqref{pgc}. Then there is a constant $C$ and a positive integer $J_0$ such that for all $j_0 \geq J \geq J_0$ and for all finite Borel measures $\mu$ on the unit ball,
\[ \int_a^b \int_{B\left(0, 2^{3J}\right)} \left\lvert \widehat{\mu_g}\left( \eta_1 \gamma'(\theta)+ \eta_2 \left( \gamma \times \gamma'\right)(\theta) \right)\right\rvert^2 \, d\eta \, d\theta \leq C2^{CJ}\mu\left( \mathbb{R}^3 \right)^2. \]
\end{lemma}
\begin{proof} This follows from the trivial bound on each $M_T(\mu \ast \phi_{j_0})$, and the rapid decay of $\mathcal{F}\left(M_T(\mu \ast \phi_{j_0})\right)$ outside $1.01\tau(T)$, for each $\tau$. \end{proof}

\begin{lemma} \label{goodpart} Suppose that $\gamma: [a,b] \to S^2$ is a $C^3$ unit speed curve satisfying \eqref{angleassumption}, \eqref{meanvalue} and \eqref{pgc} on $[a,b]$. Let $\left[\widetilde{a},\widetilde{b}\right] \subseteq (a,b)$ be such that $\left\lvert \widetilde{b}-\widetilde{a}\right\rvert \leq \sigma$, where $\sigma$ is a constant that works in Lemma~\ref{taudistance}. Let $\alpha \in [0,3]$ and $\epsilon>0$. If $\alpha_0' \in \left(0,  \alpha_0(\alpha,1, \gamma\restriction_{\left[ \widetilde{a}, \widetilde{b} \right] })\right)$, then there exists $\delta_0 > 0$, and for any $\delta \in (0, \delta_0]$ a $J_0 \geq 0$, such that for all $j_0 \geq 3J$ and $J \geq J_0$ with $J \in [(j_0\epsilon)/1000, 1000j_0\epsilon]$ and for all Borel measures $\mu$ on the unit ball with $c_{\alpha}(\mu) \leq 1$, 
\begin{equation} \label{goal2} \int_{\widetilde{a}}^{\widetilde{b}} \int \left\lvert \pi_{\theta\#} \mu_g \right\rvert^2 \, d\mathcal{H}^2 \, d\theta  \leq \mu\left( \mathbb{R}^3 \right) 2^{j_0\left( \max\left\{0,  2- \alpha, \frac{3}{2} -  \frac{\alpha}{3} - \frac{\alpha_0'}{3} \right\} + 10^4\epsilon \right)}. \end{equation}
\end{lemma} 
\begin{proof} By Plancherel, 
\begin{multline} \label{localise} \int_{\widetilde{a}}^{\widetilde{b}} \int \left\lvert \pi_{\theta \#} \mu_g \right\rvert^2 \, d\mathcal{H}^2 \, d\theta  \\
= \int_{\widetilde{a}}^{\widetilde{b}} \int_{B\left(0, 2^{j_0(1+\delta) } \right) } \left\lvert \widehat{\mu_g}\left( \eta_1 \gamma'(\theta)+ \eta_2 \left( \gamma \times \gamma'\right)(\theta) \right)\right\rvert^2 \, d\eta \, d\theta \\
+\int_{\widetilde{a}}^{\widetilde{b}} \int_{\mathbb{R}^2 \setminus B\left(0, 2^{j_0(1+\delta) } \right) }  \left\lvert \widehat{\mu_g}\left( \eta_1 \gamma'(\theta) + \eta_2 \left( \gamma \times \gamma'\right)(\theta) \right)\right\rvert^2 \, d\eta \, d\theta. \end{multline}
The inequality 
\[ \int_{\widetilde{a}}^{\widetilde{b}} \int_{\mathbb{R}^2 \setminus B\left(0, 2^{j_0(1+\delta) } \right) }  \left\lvert \widehat{\mu_g}\left( \eta_1 \gamma'(\theta) + \eta_2 \left( \gamma \times \gamma'\right)(\theta) \right)\right\rvert^2 \, d\eta \, d\theta \lesssim \mu\left( \mathbb{R}^3 \right)^2 \]
follows straightforwardly from the rapid decay of $\widehat{\phi_{j_0}}$ outside $B\left(0, 2^{j_0} \right)$; see \cite[pp.~13--14]{harris20} for a more detailed calculation of a similar inequality.  The other term in \eqref{localise} can be written as
\begin{multline} \label{splitting} \int_{\widetilde{a}}^{\widetilde{b}} \int_{B\left(0, 2^{j_0(1+\delta) } \right) } \left\lvert \widehat{\mu_g}\left( \eta_1 \gamma'(\theta)+ \eta_2 \left( \gamma \times \gamma'\right)(\theta) \right)\right\rvert^2 \, d\eta \, d\theta \\
 = \int_{\widetilde{a}}^{\widetilde{b}} \int_{\{ \left\lvert \eta_1\right\rvert \geq \left\lvert \eta_2\right\rvert^{1-\epsilon} \} \cap B\left(0, 2^{j_0(1+\delta) } \right) } \left\lvert \widehat{\mu_g}\left( \eta_1 \gamma'(\theta)+ \eta_2 \left( \gamma \times \gamma'\right)(\theta) \right)\right\rvert^2 \, d\eta \, d\theta \\
+ \int_{\widetilde{a}}^{\widetilde{b}} \int_{\{ \left\lvert \eta_1\right\rvert < \left\lvert \eta_2\right\rvert^{1-\epsilon} \} \cap B\left(0, 2^{j_0(1+\delta) } \right) } \left\lvert \widehat{\mu_g}\left( \eta_1 \gamma'(\theta)+ \eta_2 \left( \gamma \times \gamma'\right)(\theta) \right)\right\rvert^2 \, d\eta \, d\theta. \end{multline}
By Lemma~\ref{easypart}, the first term satisfies 
\begin{multline} \label{energymu} \int_{\widetilde{a}}^{\widetilde{b}} \int_{\{ \left\lvert \eta_1\right\rvert \geq \left\lvert \eta_2\right\rvert^{1-\epsilon} \} \cap B\left(0, 2^{j_0(1+\delta) } \right) } \left\lvert \widehat{\mu_g}\left( \eta_1 \gamma'(\theta)+ \eta_2 \left( \gamma \times \gamma'\right)(\theta) \right)\right\rvert^2 \, d\eta \, d\theta\\
 \lesssim \mu\left( \mathbb{R}^3 \right)^2 +  \int_{\widetilde{a}}^{\widetilde{b}} \int_{\{ \left\lvert \eta_1\right\rvert \geq \left\lvert \eta_2\right\rvert^{1-\epsilon} \} \cap B\left(0, 2^{j_0(1+\delta) } \right)} \left\lvert \widehat{\mu}\left( \eta_1 \gamma'(\theta)+ \eta_2 \left( \gamma \times \gamma'\right)(\theta) \right)\right\rvert^2 \, d\eta \, d\theta. \end{multline} 
The change of variables 
\[ \xi = \xi(\eta, \theta) =  \eta_1 \gamma'(\theta)+ \eta_2 \left( \gamma \times \gamma'\right)(\theta) \]
has Jacobian 
\begin{align}  \label{stp1} \left\lvert \frac{ \partial(\xi_1, \xi_2, \xi_3) }{\partial(\eta_1, \eta_2, \theta) }(\eta_1, \eta_2, \theta) \right\rvert &=  \left\lvert \eta_1\right\rvert \left\lvert \det\left(\left(\gamma\times \gamma'\right)(\theta), \gamma'(\theta),  \gamma''(\theta) \right) \right\rvert \\
\label{stp2} &= \left\lvert \eta_1\right\rvert \left\lvert \left\langle \left[\gamma' \times \left( \gamma \times \gamma' \right)\right](\theta), \gamma''(\theta) \right\rangle \right\rvert \\
\notag &= \left\lvert \eta_1\right\rvert \left\lvert \left\langle \gamma(\theta), \gamma''(\theta) \right\rangle \right\rvert \\
\label{unitspeed} &= \left\lvert \eta_1\right\rvert. \end{align} 
The line \eqref{unitspeed} above follows from the assumption that $\gamma$ is a curve in $S^2$ with unit speed, whilst \eqref{stp1} and \eqref{stp2} use the scalar triple product formula $\det(a,b,c) = \langle a, b \times c \rangle$ and the identity 
\[ \left(\gamma \times \gamma'\right)' = -\det\left(\gamma, \gamma', \gamma''\right) \gamma'. \]
Applying this change of variables to \eqref{energymu} gives
\begin{align*} \eqref{energymu}  &\lesssim   \mu\left( \mathbb{R}^3 \right)^2 +  \int_{B\left(0, 2^{j_0(1+\delta) } \right)} \left\lvert \xi \right\rvert^{\epsilon -1} \left\lvert \widehat{\mu}(\xi)\right\rvert^2 \, d\xi \\
&\leq   \mu\left( \mathbb{R}^3 \right)^2 +\begin{cases} 2^{j_0(1+\delta)(2+2\epsilon-\alpha) }I_{\alpha-\epsilon}(\mu) & \alpha \leq 2 \\
 2^{j_0(1+\delta)2\epsilon} I_{2-\epsilon}(\mu) & \alpha > 2 \end{cases} \\
&\lesssim  \mu\left( \mathbb{R}^3 \right)^2 +  2^{j_0(1+\delta)(\max\{0,2-\alpha\} + 2\epsilon) }\mu\left( \mathbb{R}^3 \right), \end{align*} 
which is much smaller than the right hand side of \eqref{goal2}. This bounds the first term in \eqref{splitting}. 

 It remains to bound the second term in \eqref{splitting}. This satisfies
\begin{align*} &\int_{\widetilde{a}}^{\widetilde{b}} \int_{\{ \left\lvert \eta_1\right\rvert < \left\lvert \eta_2\right\rvert^{1-\epsilon} \} \cap B\left(0, 2^{j_0(1+\delta) } \right) } \left\lvert \widehat{\mu_g}\left( \eta_1 \gamma'(\theta)+ \eta_2 \left( \gamma \times \gamma'\right)(\theta) \right)\right\rvert^2 \, d\eta \, d\theta \\
&\quad \leq \sum_{j \in [3J, j_0(1+\delta)] } \sum_{k \in [j\epsilon, j] } \int_{\widetilde{a}}^{\widetilde{b}}\int_{A_{j,k}} \left\lvert \widehat{\mu_g}\left( \eta_1 \gamma'(\theta)+ \eta_2 \left( \gamma \times \gamma'\right)(\theta) \right)\right\rvert^2 \, d\eta \, d\theta  \\
&\qquad + \int_{\widetilde{a}}^{\widetilde{b}} \int_{B\left(0, 2^{3J}\right)} \left\lvert \widehat{\mu_g}\left( \eta_1 \gamma'(\theta)+ \eta_2 \left( \gamma \times \gamma'\right)(\theta) \right)\right\rvert^2 \, d\eta \, d\theta \\
&\quad \lesssim \mu\left( \mathbb{R}^3 \right) 2^{j_0\left(\max\left\{ 0, 2-\alpha,  \frac{3}{2} -\frac{\alpha}{3} - \frac{\alpha_0'}{3} \right\}+200\epsilon \right)},
 \end{align*} 
by Lemmas~\ref{kneqjcase} and \ref{trivial}. This covers the final case and finishes the proof of the lemma. \end{proof} 

\begin{lemma} \label{conversion}
\label{alphastars} Let $\gamma: [a,b] \to S^2$ be a $C^3$ unit speed curve with $\det(\gamma, \gamma', \gamma'') \neq 0$ on $[a,b]$, let $\left[ \widetilde{a},\widetilde{b}\right] \subseteq (a,b)$, and let $\alpha \in (0,3]$. Then for any $\alpha^*>0$ with 
\[ \alpha^* < \min\left\{2, \alpha , \frac{\alpha}{3} + \frac{1}{2} + \frac{\alpha_0(\alpha,1,\gamma\restriction_{\left[\widetilde{a},\widetilde{b} \right]})}{3} \right\} , \]
there exist $\delta'',C>0$ such that 
\begin{equation} \label{recursion} \int_{\widetilde{a}}^{\widetilde{b}} \left( \pi_{\theta \#} \mu \right)\left( \bigcup_{D \in \mathbb{D}_{\theta}} D \right) \, d\theta \leq C \mu\left( \mathbb{R}^3 \right) R^{-\delta''}, \end{equation}
for all Borel measures $\mu$ on the unit ball with $c_{\alpha}(\mu) \leq 1$, for any $R \geq 1$, and for any collection of sets $\{\mathbb{D}_{\theta} : \theta \in \left[\widetilde{a},\widetilde{b} \right] \}$ such that the integrand of \eqref{recursion} is measurable, where, for each $\theta \in \left[ \widetilde{a}, \widetilde{b} \right]$, $\mathbb{D}_{\theta}$ is a disjoint set of at most $\mu(\mathbb{R}^3)R^{\alpha^*/2}$ discs in $\pi_{\theta}(\mathbb{R}^3)$ of radius $R^{-1/2}$. 

As a corollary,  $\alpha_0(\alpha,1,\gamma\restriction_{\left[\widetilde{a},\widetilde{b} \right]}) \geq \min\left\{ 2, \alpha, \frac{\alpha}{2} + \frac{3}{4} \right\}$. 
\end{lemma}
\begin{proof} By localisation it may be assumed that $\gamma$ satisfies \eqref{angleassumption}, \eqref{meanvalue} and \eqref{pgc} on $[a,b]$, and that $\left\lvert b-a\right\rvert \leq \sigma$ where $\sigma$ is a constant that works in Lemma~\ref{taudistance}. Let $\alpha, \alpha^*, R, \mu$ and the sets $\mathbb{D}_{\theta}$ be given. Define 
\[ \epsilon = \frac{ \min\left\{ \alpha, \frac{\alpha}{3} + \frac{1}{2} + \frac{\alpha_0}{3},2 \right\} - \alpha^* }{10^{10}}. \]
By the ``good-bad'' decomposition (Lemma~\ref{decomp}) with $\alpha_0' := \alpha_0-\epsilon$ and $\delta \ll \epsilon$, $J$ defined by $R^{\epsilon} \in \left[2^J, 2^{J+1} \right)$ and $j_0$ defined by $R^{1/2} \in \left[2^{j_0}, 2^{j_0+1} \right)$, 
\begin{align} \notag  &\int_{\widetilde{a}}^{\widetilde{b}} \left( \pi_{\theta \#} \mu \right)\left( \bigcup_{D \in \mathbb{D}_{\theta}} D \right) \, d\theta \\
\notag &\quad \lesssim \int_{\widetilde{a}}^{\widetilde{b}} \int \left\lvert \pi_{\theta \#} \mu_b \right\rvert d\mathcal{H}^2 \, d\theta + \sup_{\theta \in \left[\widetilde{a}, \widetilde{b}\right] } \mathcal{H}^2\left( \bigcup_{D \in \mathbb{D}_{\theta}} D \right)^{1/2} \left(\int_{\widetilde{a}}^{\widetilde{b}} \int \left\lvert \pi_{\theta \#} \mu_g \right\rvert^2 d\mathcal{H}^2 \, d\theta\right)^{1/2} \\
\label{dichotomy} &\quad \lesssim \int_{\widetilde{a}}^{\widetilde{b}} \int \left\lvert \pi_{\theta \#} \mu_b \right\rvert d\mathcal{H}^2 \, d\theta + \mu\left( \mathbb{R}^3 \right)^{1/2} R^{\frac{\alpha^*}{4} - \frac{1}{2} } \left(\int_{\widetilde{a}}^{\widetilde{b}} \int \left\lvert \pi_{\theta \#} \mu_g \right\rvert^2 d\mathcal{H}^2 \, d\theta\right)^{1/2}.  \end{align}
By Lemma~\ref{badpart}, 
\begin{equation} \label{badbound2} \int_{\widetilde{a}}^{\widetilde{b}} \int \left\lvert \pi_{\theta \#} \mu_b \right\rvert d\mathcal{H}^2 \, d\theta \leq \mu(\mathbb{R}^3) R^{-\epsilon \delta'}, \end{equation}
provided $R$ is sufficiently large.  By Lemma~\ref{goodpart},
\begin{multline} \label{goodbound} \mu\left( \mathbb{R}^3 \right)^{1/2} R^{\frac{\alpha^*}{4}- \frac{1}{2}} \left(\int_{\widetilde{a}}^{\widetilde{b}} \int \left\lvert \pi_{\theta \#} \mu_g \right\rvert^2 d\mathcal{H}^2 \, d\theta\right)^{1/2} \\
 \lesssim \mu(\mathbb{R}^3) R^{ \frac{1}{2}\max\left\{\alpha^*-2,  \alpha^* - \alpha, \alpha^* - \frac{1}{2} - \frac{\alpha}{3} -\frac{\alpha_0}{3}\right\} + 10^5\epsilon }. \end{multline}
Applying \eqref{badbound2} and \eqref{goodbound} to \eqref{dichotomy} yields \eqref{recursion} and proves the first part of the lemma, provided $\delta''$ is chosen sufficiently small, $R$ is taken sufficiently large, and the constant $C(\alpha, \alpha^*,\gamma)$ is then taken large enough to handle the small values of $R$. 

The last part of the lemma follows directly from the first part, and from Definition~\ref{alpha0defn} which defines $\alpha_0$. \end{proof} 

\begin{proof}[Proof of Theorem~\ref{mainthm}] Assume without loss of generality that $\dim A >0$ and that $A$ is a subset of the unit ball. Let $\epsilon >0$ be small, let $\alpha = \dim A - \epsilon$ and (using Frostman's lemma) let $\mu$ be a nonzero, finite Borel measure on $A$ with $c_{\alpha}(\mu) \leq 1$. Suppose that $E \subseteq [a+\epsilon,b-\epsilon]$ is a compact set such that
\[ \dim \pi_{\theta}\supp \mu \leq s := \min\left\{2, \alpha, \frac{\alpha}{2} + \frac{3}{4} \right\} - 10\epsilon, \]
for every $\theta \in E$. Let $\varepsilon>0$ be small. For each $\theta \in E$, let\footnote{Issues of measurability will be ignored since they can be easily adjusted for.} $\left\{ B\left(\pi_{\theta}(x_j(\theta)), r_j(\theta) \right) \right\}_{j=1}^{\infty}$ be a covering of $\pi_{\theta} \supp \mu$ by discs of dyadic radii smaller than $\varepsilon$, with each $x_j(\theta) \in \supp \mu$, such that 
\begin{equation} \label{notubes} \sum_j r_j(\theta)^{s+\epsilon} < 1. \end{equation}
It may be additionally assumed that $\pi_{\theta}(x_j(\theta)) \in \pi_{\theta} \supp \mu$ for every $j$ and $\theta$. For each $\theta \in E$ and each $k \geq \left\lvert \log_2 \varepsilon\right\rvert$ let 
\[ D_k(\theta) = \bigcup_{j : r_j(\theta) = 2^{-k} } B\left(\pi_{\theta}(x_j(\theta)), r_j(\theta) \right). \]
Then for each $\theta \in E$,
\[ 1 \leq \sum_{k \geq \left\lvert \log_2 \varepsilon\right\rvert} \left(\pi_{\theta \#} \mu\right)(D_k(\theta)). \]
By the Besicovitch covering theorem, for each $\theta \in E$ there is a disjoint subcollection $\left\{ B\left(\pi_{\theta}(x_j(\theta)), r_j(\theta) \right) \right\}_{j\in I}$ and corresponding subsets $D_k(\theta) ' \subseteq D_k(\theta)$, such that 
\[ 1 \lesssim \sum_{k \geq \left\lvert \log_2 \varepsilon\right\rvert} \left(\pi_{\theta \#} \mu\right)(D_k'(\theta)). \]
and hence 
\[ \mathcal{H}^1(E) \lesssim \sum_{k \geq \left\lvert \log_2 \varepsilon\right\rvert} \int_E  \left(\pi_{\theta \#} \mu\right)(D_k'(\theta)) \, d\theta. \] 
By \eqref{notubes}, for each $\theta$ and $k$ the set $D_k(\theta)'$ is the union of at most $2^{k(s+\epsilon)}$ disjoint discs of radius $2^{-k}$. By Lemma~\ref{alphastars}, there exists a $\delta>0$ independent of $\varepsilon$ such that
\begin{align*} \mathcal{H}^1(E) 
&\lesssim   \sum_{k \geq \left\lvert \log_2 \varepsilon\right\rvert} 2^{-k\delta} \\
&\lesssim \varepsilon^{ \delta}. \end{align*} 
Letting $\varepsilon \to 0$ gives $\mathcal{H}^1(E) = 0$. Hence 
\[ \dim \pi_{\theta}(A) \geq  \dim \pi_{\theta}\supp \mu \geq \min\left\{2,  \alpha, \frac{\alpha}{2} + \frac{3}{4} \right\} - 10\epsilon \]  
for a.e.~$\theta \in [a+\epsilon,b-\epsilon]$. The theorem follows by letting $\epsilon \to 0$ along a countable sequence. \end{proof}  

\section{Further improvement and related problems} \label{discussion}

A related problem is the family of projections $\rho_{\theta}(x) = \langle x, \gamma(\theta) \rangle \gamma(\theta)$ onto lines, where $\gamma$ is a smooth curve in $S^2$ with $\det(\gamma, \gamma', \gamma'')$ nonvanishing. For non-great circles this was resolved by Käenmäki-Orponen-Venieri in \cite{kaenmaki}. I do not know if the method of proof here would also work on this problem; I would guess that at least a different kind of refined Strichartz inequality would be needed, with ``slabs'' in place of ``tubes''. One application, due to Liu~\cite{liu7}, of the Käenmäki-Orponen-Venieri projection theorem is to give the sharp lower bound of 3 for the Hausdorff dimension of Kakeya sets in the first Heisenberg group $\mathbb{H}$. It would be interesting if Theorem~\ref{mainthm} could be analogously applied to generalised Besicovitch sets  in $\mathbb{H}$ (e.g.~sets containing a left translate of every vertical subgroup). 

I do not know if Theorem~\ref{mainthm} holds with $C^3$ replaced by $C^2$.  The only step in the proof of Theorem~\ref{mainthm} which seems to make crucial use of the $C^3$ assumption is the decoupling theorem for generalised cones, though the statement and some steps in the proof of the refined Strichartz inequality would likely become more technical if $\gamma$ were only assumed to be $C^2$. As far as I am aware, the decoupling theorem for generalised cones has only been proved in the literature with $C^3$ assumptions (see e.g.~\cite{pramanikseeger}).

Although the proof of Theorem~\ref{mainthm} has some similarities with the proof of the lower bound for the distance set problem in \cite{GIOW}, it makes crucial use of the fact that the projections $\pi_{\theta}$ are linear. Since the distance function is nonlinear, the recursive method of bounding the ``bad'' part in the proof of Theorem~\ref{mainthm} does not seem to work on the distance set problem. 

\begin{sloppypar} Another problem on restricted projections comes from the family of maps $S_g(x,y) = x-gy$ from $\mathbb{R}^{2n}$ to $\mathbb{R}^n$, where $n \geq 2$ is fixed and $g$ ranges over $O(n)$. In \cite{mattila2020} Mattila proved that 
\begin{equation} \label{dimbound} \dim S_g(A) \geq \max\{ \min\{\dim A,n-1\}, \min\{\dim A -1, n \} \}, \qquad \text{a.e.~$g \in O(n)$}, \end{equation}
and asked whether this can be replaced by $\dim S_g(A) \geq \min\{ n , \dim A\}$ for a.e.~$g \in O(n)$. A counterexample (obtained with help from A.~Barron) is the set $A= \{(x,x): x \in \mathbb{R}^n\}$; since every $g \in O(n)$ with $\det g = (-1)^{n+1}$ has 1 as an eigenvalue, the set $S_g(A)$ has dimension at most $n-1$ whenever $\det g = (-1)^{n+1}$. By taking direct sums with Cantor subsets of the plane $\{(x,-x): x \in \mathbb{R}^n\}$, this can be modified to show that \eqref{dimbound} is sharp for all values of $\dim A$. It would be interesting to know whether a better lower bound is possible if the requirement ``a.e.~$g \in O(n)$'' is weakened to ``with probability at least 1/2''. \end{sloppypar}

\end{document}